\documentclass[12pt,twoside]{amsart}

\usepackage[usenames]{color}

\usepackage{latexsym}
\usepackage{amssymb}
\usepackage{amsfonts}
\usepackage{amstext}
\usepackage{multicol}
\usepackage{amsmath}
\usepackage{amsthm}
\usepackage{setspace}
\usepackage{appendix}
\usepackage{bm}

\newtheorem{theorem}{Theorem}[section]
\newtheorem{lemma}[theorem]{Lemma}
\newtheorem{corollary}[theorem]{Corollary}
\newtheorem{proposition}[theorem]{Proposition}
\newtheorem{definition}[theorem]{Definition}
\newtheorem{remark}[theorem]{Remark}
\newtheorem{assumption}[theorem]{Assumption}

\def\P{{\mathbb P}}
\def\E{{\mathbb E}}
\def\e{{\mathcal E}}
\def\Z{{\mathbb Z}}
\def\V{{\Vert}}

\def\R{{\mathbb R}}
\def\HS{{\rm HS}}
\def\L{{\mathcal L}}
\def\A{{\mathcal A}}
\def\N{{\mathcal N}}
\def\M{{\mathcal M}}
\def\vv{\varepsilon}
\def\span{{\rm span}}
\def\tr{{\rm Tr}}
\def\comp{{\rm Comp}}
\def\incomp{{\rm Incomp}}
\def\dom{{\rm Dom}}
\def\sparse{{\rm Sparse}}
\def\dist{{\rm dist}}

\begin{document}
\title{Investigate invertibility of sparse symmetric matrix}
\date{}
\author{Feng Wei}
\thanks{Partially supported by M. Rudelson's  NSF Grant DMS-1464514,  and USAF Grant FA9550-14-1-0009. }

\maketitle

\begin{abstract}
In this paper, we investigate the invertibility of sparse symmetric matrices. We will show that for  an $n\times n$ sparse  symmetric random matrix $A$ with $A_{ij} = \delta_{ij} \xi_{ij}$ is invertible with high probability. 
Here, $\delta_{ij}$s, $i\ge j$ are i.i.d. Bernoulli random variables with $\P \left( \xi_{ij}=1 \right) =p \ge n^{-c}$, $\xi_{ij}, i\ge j$ are i.i.d.  random variables with mean 0, variance 1 and finite forth moment $M_4$, and $c$ is constant depending on $M_4$. 
More precisely, 
$$
 s_{\rm min} (A) > \vv \sqrt{\frac{p}{n}}. 
$$
with high probability.
\end{abstract}

\section{Introduction}

Singular values and eigenvalues are both important characteristics of random matrices and their magnitude agrees on symmetric matrices. Non-asymptotic random matrix theory studies  spectral properties of random matrices, that is to provides probabilistic bounds for  singular values, eigenvalues, etc., for random matrices of a large fixed size. In the non-asymptotic viewpoint, study of singular values are more motivated due to  geometric problems in high dimensional Euclidean spaces.

Recall that for  an $n\times n$ real matrix, the singular values $s_k(A)$ of $A$, where $k=1,2,\cdots, n$, are the eigenvalues of $\sqrt{A^T A}$ arranged in non-increasing order. Among all the singulars, the two extreme ones are of the most importance. When we view  matrix $A$ as a linear operator $\R^n\rightarrow\R^n$, we may want control its behavior by finding or giving  useful upper and lower bounds on $A$. Such bounds are provided by the smallest and largest singular values of $A$ denoted as $s_{\rm min}(A)$ and $s_{\rm max}(A)$. The extreme singular values are also referred as  the operator norms of the linear operators $A$ and $A^{-1}$ between Euclidean spaces, that is to say $s_{\rm min}(A)= 1/ \V A^{-1}\V $ and $s_{\rm max}(A) = \V A \V$.

Due to the geometric interpretation, understanding the behavior of extreme singular values of random matrices are important in many applications. For instance, in computer science and numerical linear algebra, the condition number $s_{\rm max}(A)/s_{\rm min}(A)$ is widely used to measure stability or efficiency of algorithms as the example we give in early section. In geometric functional analysis, probabilistic construction of linear operators using random matrices often depend on good bounds on the norms of these operators and their inverses \cite{AGA}. In statistics,  applications of extreme singular values can be found  from the analysis of sample covariance matrices $ A^TA$ \cite{RVbookintrotorm}.

It is worth mentioning that many non-asymptotic  results are known under a somewhat stronger sub-gaussian moment assumption on the entries of  $A$, which requires their distribution to decay as fast as the normal random variable:
\begin{definition}
(sub-gaussian random variables). A random variable $X$ is sub-gaussian
if there exists $K > 0$ called the sub-gaussian moment of X such that
$$
P(|X| > t) \le  2\exp(-t^2/K^2) {\rm \ for\ } t >0.
$$
\end{definition}

Many classical random variables are actually sub-gaussian, such as Gaussian random variables, Bernoulli random variables, Bounded random variables, etc..

Von Neumann and Goldstine conjectured that with high probability $s_{\rm min}(a)\sim n^{-1/2}$ and $s_{\rm max}(a)\sim n^{1/2}$ \cite{NAtheoryofRM}. The upper bound on largest singular value was established earlier but the lower bound of smallest singular value remained open for decades. Progress has been made by Smale,  Edelman and Szarek in the Gaussian matrices case \cite{SSeaa,AEecnrm,SScnrm}. However, their approaches do not work for matrices other than Gaussian as they depend on explicit formula for the joint density of the singular values. 

The first polynomial bound of quantitative invertibility was obtained in \cite{InvofRM} by M. Rudelson, where it was proved that the smallest singular value of a square i.i.d. sub-gaussian matrix is bounded below by $n^{-3/2}$ with high probability. Later an almost sharp bound was proved by M. Rudelson and R. Vershynin in \cite{LOandInvofRM} up to a constant factor for general random matrices. 

\begin{theorem}\label{ssvintro}
{\rm (Smallest singular value of square random matrices)}. Let $A$ be
an $n \times n$ random matrix whose entries are independent and identically distributed
sub-gaussian random variables with zero mean and unit variance. Then
$$
\P\left( s_{\rm min}(A)\le \vv n^{-1/2} \right)\le C\vv + c^n,\ \vv \ge 0
$$
where $C>0,c\in (0,1)$ depend only on the sub-gaussian moment of the entries.
\end{theorem}

The theory and result was later extended to rectangular random matrices of arbitrary dimensions in \cite{SMstSvRect}.
\begin{theorem}\label{ssvrintro}{\rm  (Smallest singular value of rectangular random matrices).}
Let $G$ be an $N\times n$ random matrix, $N\ge n$, whose elements are independent copies of a centered sub-gaussian random variable with unit variance. Then for every $\vv >0$, we have
$$
\P \left( s_n(G) \le \vv \left( \sqrt{N}-\sqrt{n-1}\right)\right)\le (C\vv)^{N-n+1}+e^{-C'N}
$$
where $C,C'>0$ depend (polynomially) only on the sub-gaussian moment $K$.
\end{theorem}

These results was extended and improved  in a number of papers, including \cite{TaoVuAnnals,TVsmallestSVdist,LOandInvofRM,InvofsparsenonHerm,Invofheavytail,SMstSvRect,Invofsym}. However, first such result sparse matrix appeared until \cite{InvofsparsenonHerm}, where Basak and Rudelson proved that for a non-Hermitian i.i.d. sparse matrix,
\begin{equation}
\begin{array}{rl}
& \P \left\{ s_{{\rm min}} (A) \ge C\vv \displaystyle\exp\left(
-c\frac{\log(1/p_n)}{\log(np_n)}
\right)
\displaystyle\sqrt{\frac{p_n}{n}}\bigcap 
{\V A \V \le C\sqrt{pn}}
\right\}\\
\le & \vv +\exp(-cnp_n)
\end{array}
\end{equation}
where $\P(a_{ij}\neq 0)=p_n$. One may notice that for $p_n \ge n^{-c}$, where $0<c<1$, the above result of Basak and Rudelson implies an upper bound on condition number. That is to say
$$
\sigma(A_n) := \frac{s_{{\rm max}}(A)}{s_{{\rm min}}(A_n)} \le n
$$
with high probability. This generalized the optimal upper bound on condition number for non sparse random matrices. So it is a nature question to ask, whether one can use the similar technique to develop the invertibility for sparse symmetric matrices. 

One  contribution of  \cite{InvofsparsenonHerm} is a combinatorial approach to address the sparsity in estimating the norm of $Ax$ for a sparse matrix $A$ and sparse vector $x$. The combinatorial lemma is generalizable in symmetric matrices case which makes it possible to  prove quantitative invertibility for symmetric sparse matrix together with a decoupling method in \cite{Invofsym}. This work is motivated by the above result of non-Hermitian sparse matrices of A. Basak and M. Rudelson and the paper of R. Vershynin for non-sparse symmetric matrices, see \cite{Invofsym}. Without special notice, we always assume the following for our random matrix $A_n$:

\begin{assumption}\label{assump}
$A_n=\{a_{i,j}\}_{i,j=1}^n$ is an $n\times n$ symmetric random matrix with i.i.d entries on the upper triangular part, and $a_{i,j}=\xi_{ij}\delta_{ij}$. Here $\delta_{ij}$s are i.i.d. Bernoulli random variables with $\P(\delta_{ij}=1)=p_n$. $\xi_{ij}$s are i.i.d. random variables with mean zero, variance 1 and fourth moment bounded by $M_4^4$. 
\end{assumption}

\begin{remark}
{\rm The dependence of $c_p$ on $M_4$ is tracked in the the proof.}
\end{remark}
\begin{remark}
{\rm Through out the paper, we are going to call $p_n$ the sparsity level of $A$.}
\end{remark}
\begin{remark}
{\rm For the ease of writing, hereafter, we will often drop the sub-script in $A_n, p_n$, write $A,p$ instead.  But please have it in mind that the sparsity level will depend on $n$.}
\end{remark}

Our proof will also use an upper bound for operator norm. For convenience, throughout the proof, we denote $\e_{op}$ as the event that $\V A \V\le C_{op}\sqrt{pn}$.

Our main theorem is the following:
\begin{theorem}\label{mth}{\rm (Smallest singular value for sparse symmetric matrices.)}
For $A$ satisfies Assumption \ref{assump} and  $p\ge n^{-c_p}$, where $c_p$ is a constant depending only on $M_4,C_{op}$, one has
$$
\P\left( s_n(A)\le \vv \sqrt{\frac{p}{n}} \cap \e_{op} \right) \le C_{\ref{mth}}\vv^{1/9} + e^{-n^{c_{\ref{mth}}}}.
$$
Here $C_{\ref{mth}},c_{\ref{mth}}>0$ depend only on $M_4$ and $C_{op}$.
\end{theorem}

\begin{remark}
{\rm Theorem \ref{mth} can be also generalized to the case  $A$ is replaced by $A+D$ where $D$ is a diagonal matrix  and $\V D \V = O(\sqrt{pn})$. For simplicity, we do not include the proof in this paper, see \cite{InvofsparsenonHerm} for more details.}
\end{remark}

Recall that for a random variable $Z$  on a probability space $(\Omega, \mathcal{A},\P)$. The sub-gaussin norm or $\psi_2$-norm of $Z$ is defined as
$$
\V Z \V_{\psi_2}:=\inf\left\{ \lambda>0: \E \exp\left( \frac{| Z |}{\lambda} \right)^2   \le 2\right\}.
$$
A random variable is called sub-gaussian if it has finite sub-gaussian norm. For properties of sub-gaussian random variables, see \cite{LecnotesRMMR}. 
For sparse symmetric matrix with $\xi_{ij}$s are sub-gaussian, we have the following result about spectral norm. 

\begin{theorem}\label{upbdforop}
There exists $C'_{\ref{upbdforop}}\ge 1$ such that the following holds. Let $n\in \N$ and $p \in (0,1]$ be such that $p \ge C'_{\ref{upbdforop}} \frac{\log n}{n}$. Let $A_n$ be a random matrix as in Assumption \ref{assump}. Moreover, we require $\xi_{ij}$ to be sub-Gaussian random variables in the assumption. Then there exist positive constants $C_{\ref{upbdforop}},c_{\ref{upbdforop}}$ depending on the sub-Gaussian norm of ${\xi_{ij}}$, such that
$$
\P\left(
\V A_n \V \ge C_{\ref{upbdforop}}\sqrt{np_n}) \le \exp(-c_{\ref{upbdforop}}np
\right).
$$
\end{theorem}
Theorem \ref{mth} and \ref{upbdforop} together give us the following result:
\begin{corollary}
\label{mthcor}{\rm (Smallest singular value for sparse symmetric sub-gaussian matrices.)}
For $A$ as in Theorem \ref{mth} and moreover $\xi_{ij}$s are sub-gaussian random variables, one has
$$
\P\left( s_n(A)\le \vv \sqrt{\frac{p}{n}} \right) \le C_{\ref{mthcor}}\vv^{1/9} + e^{-n^{c_{\ref{mthcor}}}} + \exp(-c'_{\ref{mthcor}}np).
$$
Here $C_{\ref{mthcor}},c_{\ref{mthcor}},c'_{\ref{mthcor}}>0$ depend only on the sub-gaussian norm.
\end{corollary}

\begin{remark}
{\rm For sparse sub-gaussian matrix, above theorems directly yield a bound on the condition number that $n \gtrsim \sigma(A):=\frac{s_{max}(A)}{s_{min}(A)}$ with high probability.}
\end{remark}

{\bf Outline of paper.}

\begin{itemize}

\item In Section \ref{SecNandP}, we recall the necessary concepts and some technical lemmas. We also recall the method of separating compressible and incompressible vectors (see \cite{LOandInvofRM}) in Section \ref{SecNandP}.

\item In Section \ref{SecIoverC}, we bound $\V Ax \V_2$ over compressible vectors. The method we used to bound the infimum over compressible vectors for sparse matrix is invented in Section 3 in  \cite{InvofsparsenonHerm}. 

\item In Section \ref{SecIoverInc}, \ref{SecStruc} and \ref{SecSBquad} we bound $\V Ax \V_2$ over incompressible vectors. In Sec \ref{SecIoverInc} we recall the definition of LCD and regularized LCD and reduce the infimum to a distance problem which can be written as a quadratic form, see \cite{Invofsym}. In Section \ref{SecStruc}, we prove the structure theorem for large LCD vectors which is an analog of Theorem 7.1 in \cite{InvofRM}. In Section \ref{SecSBquad}, we estimate the distance problem using the decoupling technique in \cite{InvofRM}.

\item In Section \ref{SecMproof}, we combine the estimate for compressible and incompressible part to prove our main theorem. 

\item In Section \ref{SecSpectNorm}, we prove an upper bound of the spectral norm for sparse symmetric sub-gaussian matrix which is an analog of Theorem 1.7 in \cite{InvofsparsenonHerm}.

\end{itemize}

\section{Notations and Preliminaries}\label{SecNandP}

We first explain our notations. Through out the paper $c,C,c_0,c_1,$ $c',\cdots$ denote absolute constants or constants that are going to be used only locally. These constants are different in proofs of different lemmas or theorems. Constants with double indices or letter indices are global constants, they are uniform through out the paper and we will keep track of these constants through out the paper, for example $c_{3.1},c_{3.2}',c_p$.

First, recall that
$$
s_{{\rm min}} (A_n) = \inf_{x\in S^{n-1}} \left\Vert A_n x \right\Vert.
$$
Thus, to prove Theorem \ref{mth}, we need to find a lower bound on the infimum. For dense matrices, this can be done via decomposing the unit sphere into compressible and incompressible vectors, and obtaining necessary bound on the infimum on both of these parts, see \cite{LOandInvofRM, SMstSvRect}. To carry out the argument for sparse matrices, Basak and Rudelson introduced another class of vectors which they called dominated vectors, see \cite{InvofsparsenonHerm}. 

Below, we state necessary concepts, starting with the definition of compressible and incompressible vectors, see \cite{LOandInvofRM}.

\begin{definition}
Fix $m<n$. The set of $m-$sparse vectors is given by 
$$
\sparse (m):=\left\{ x\in \R^n ||{\rm supp}(x)|\le m \right\}
$$
where $|S|$ denotes the cardinality of a set $S$. Furthermore, for any $\delta>0$, the vectors which are $\delta$-close to $m$-sparse vectors in Euclidean norm, are called $(m,\delta)-$compressible vectors. The set of all such vectors, will be denoted by $\comp(m,\delta)$. Thus
$$
\comp(m,\delta):=\left\{ x\in S^{n-1} | \exists y \in \sparse(m) \ {\rm such} \ {\rm that} \ \V x-y \V_2 \le \delta \right\}.
$$
The vectors in $S^{n-1}$ which are not compressible, are defined to be incompressible, and the set of all incompressible vectors is denoted as $\incomp(m,\delta)$.
\end{definition}

The dominated vectors  are also close to sparse vectors, but in a different sense, see \cite{InvofsparsenonHerm}.

\begin{definition}
For any $x\in S^{n-1}$, let $\pi_x:[n]\rightarrow [n]$ be a permutation which arranges the absolute values of the coordinates of $x$ in non-increasing order. For $1\le m \le m' \le n$, denote by $x_{[m:m']}\in \R^n$ the vector with coordinates
$$
x_{[m:m']}(j) = x(j) 1_{[m:m']} (\pi_x (j)).
$$
In other words, we include in $x_{[m:m']}$ the coordinates of $x$ which take places from $m$ to $m'$ in the non-increasing rearrangement. 
For $\alpha <1$ and $m\le n$ define the set of vectors with dominated tail as follows:
$$
\dom (m, \alpha) := \left\{ x\in S^{n-1} | \left\Vert x_{[m+1:n]} \right\Vert_2 \le \alpha \sqrt{m} \V x_{[m+1:n]} \V_\infty \right\}.
$$
\end{definition}
One may notice that  for $m-$sparse vectors  $x_{[m+1:n]}=0$,  thus we have $\sparse(m)\cap S^{n-1}\subset \dom (m,\alpha)$. 

Theorem \ref{mth} will be proved by first bounding the infimum over compressible and dominated vectors, and then the same for the incompressible vectors. As in \cite{InvofsparsenonHerm}, the first step is to control the infimum for sparse vectors.  To this end, we need some estimates on the small ball probability. For the estimates, recall the definition of Levy concentration function.

\begin{definition}
Let $Z$ be random variable in $\R^n$. For every $\vv >0$, the Levy's concentration function of $Z$ is defined as 
$$
\L(Z,\vv) := \sup_{u\in \R^n} \P \left( \V Z- u\V_2 \le \vv \right),
$$
where $\V \cdot \V_2 $ denotes the Euclidean norm.
\end{definition}
The following Paley-Zygmund inequality is useful on estimating Levy's concerntration function:
\begin{lemma}\label{PZineq}
If $\xi$ is a random variable with finite variance and $0 \le \theta\le 1$, then 
$$
\P (\xi > \theta \E \xi) \ge \frac{(\E \xi- \theta \E \xi)^2}{\E \xi^2}.
$$
\end{lemma}

\begin{remark}\label{delta0}
{\rm We note that there exist $\delta_0, \vv_0'\in (0,1)$, such that for any $\vv <\vv_0', \L (\xi\delta,\vv )\le 1-\delta_0 p$, where $\xi $ is a random variable with unit variance and finite fourth moment, and $\delta$ is a Ber$(p)$ random variable, independent of each other (for more details see [\cite{Invofsym}, Lemma 3.3]). }
\end{remark}


For application of Levy's concerntration function, the following tensorization lemma can be very useful to transfer bounds for the Levy concentration function from random variables to random vectors.

\begin{lemma}\label{tens}{\rm (Tensorization, Lemma 3.4 in \cite{Invofsym})}. Let $X=(X_1,\cdots,X_n)$ be a random vector in $\R^n$ with independent coordinates $X_k$.

1. Suppose there exists numbers $\vv_0\ge0$ and $L\ge 0$ such that 
$$
\L (X_k,\vv_0) \le L\vv {\rm \ for \ all \ }
\vv \ge \vv_0 {\rm \ and \ all \ } k.
$$
Then 
$$
\L (X,\vv \sqrt{n}) \le (C L\vv)^n {\rm \ for \ all \ }
\vv \ge \vv_0,
$$
where $C$ is an absolute constant .

2. Suppose there exists number $\vv >0$ and $q\in (0,1)$ such that 
$$
\L(X_k,\vv)\le q {\rm \ and \ all \ } k.
$$
There exists numbers $\vv_1 =\vv_1(\vv,q)>0$ and $q_1 =q_1(\vv,q)\in (0,1)$ such that 
$$
\L(X,\vv_1 \sqrt{n})\le q_1^n.
$$
\end{lemma}
\begin{remark}\label{tensremark}
{\rm A useful equivalent form of Lemma \ref{tens} (part 1) is the following. Suppose there exist numbers $a,b\ge 0$ such that }
$$
\L (X_k,\vv) \le a\vv+b {\rm \ for \ all \ }
\vv \ge 0 {\rm \ and \ all \ } k.
$$
{\rm Then}
$$
\L (X,\vv \sqrt{n}) \le (C (a\vv+b))^n {\rm \ for \ all \ }
\vv \ge 0,
$$
{\rm Where $C$ is an absolute constant}
\end{remark}

\section{Invertibility over compressible vectors}\label{SecIoverC}

The main theorem in this section is the following:

\begin{theorem}\label{mthcomp}
Consider $A$ satisfies \ref{assump} and $p\ge (1/4)n^{-1/3}$, then there exist $c_{\ref{mthcomp}}$, $c_{\ref{mthcomp}}'$, $c_{\ref{mthcomp}}''$,   $c_{\ref{mthcomp}}'''$, $C_{\ref{mthcomp}}>0$ depending only on $C_{op},M_4$, such that for any $p^{-1}\le M\le c_{\ref{mthcomp}}'''n$, we have for any $u\in \R^n$
\begin{equation}
\begin{array}{rl}
& \P \bigg( \exists x \in {\rm Dom} (M, C_{\ref{mthcomp}}^{-1})\cup \comp(M,c_{\ref{mthcomp}}')  \\
& \V Ax -u \V_2 \le c_{\ref{mthcomp}}'' \sqrt{np} {\rm \ and \ } \V A \V\le C_{op}\sqrt{pn} \bigg) \le \exp(-c_{\ref{mthcomp}}pn).
\end{array}
\end{equation}

\end{theorem}

\begin{remark}
{\rm Although for the purpose our our proof we do not need to bound the dominated vectors close to moderately sparse, we still work on it due to it's own interest for future work.}
\end{remark}
\begin{remark}
{\rm Theorem \ref{mthcomp} can be extended to the sparsity level of $n^{-1+c}$ for arbitrary $c$ following our framework. The reason we can't not reach $n^{-1+c}$ in Theorem \ref{mth} is due to incompressible part.}
\end{remark}

A direct proof following the paper of Vershynin \cite{Invofsym} won't work due to the sparsity phenomenon found in the sparse paper of Basak and Rudelson, see \cite{InvofsparsenonHerm}. So we need to adapt the technique for sparse matrix and deal with the symmetricity at the same time.  The proof splits into two steps as in \cite{InvofsparsenonHerm}. First, we consider vectors which are close to $(1/8p)$-sparse. The small ball probability estimate is not strong enough for such vectors. This forces us to use the method designed for sparse matrices in \cite{InvofsparsenonHerm}. We prove Lemma \ref{comblemma} which a generalized version of Lemma 3.2 in \cite{InvofsparsenonHerm} for symmetric matrix. Lemma \ref{comblemma} allows us to control $\V Ax \V_2$ for very sparse vectors without cancellation and $\vv$-net argument. For more intuition of the technique for vectors close to very sparse, see Section 3.1 in \cite{InvofsparsenonHerm}. Later, one needs to improve these estimates for vectors which are close to $M$-sparse. For such moderately sparse vectors, a better control of the Levy concentration function is available. After we obtain such estimates for sparse vectors, we extend them to compressible vectors using the standard $\vv$-net  and the union bound argument.

\subsection{Vectors close to very sparse}

Now we state a combinatorial lemma similar to Lemma  3.2 in \cite{InvofsparsenonHerm} but designed for symmetric matrices. The proof is a variant of Lemma  3.2 in \cite{InvofsparsenonHerm} to deal with the symetricity.
\begin{lemma}\label{comblemma}
Consider $A_n$ be an $n\times n$ random matrix with $a_{ij}=\delta_{ij}\xi_{ij}$  for $i\le j$   and $a_{ji} = \pm a_{ij}$ for $i > j $. Here $\delta_{ij}$ are i.i.d. Bernoulli random variables with $\P(\delta_{ij}=1)=p$, where $p\ge C\log n/n$. And $\xi_{ij}$ are independent mean zero  random variables with $\min\{ \P(\xi_{i,j}\ge c_1),\P(\xi_{i,j} \le -c_1) \}\ge c_0$. For $\kappa \in \mathbb{N} $, $s\in \{-1,1\}^{\kappa}$ and for $J,J'\subset[n]$, let $\A_c^{J,J',s}$ denote the event that satisfies the following conditions:

(i) There are at least $c \kappa pn$ rows of the matrix have non-zero entry in the columns corresponding to $J$, and all zero entries in the columns corresponding to $J'$.

(ii) Denote $I^{J,J'}$ be the indices of those $c \kappa pn$ rows. Then $I^{J,J'}\cap (J\cup J')=\emptyset$.

(iii) Suppose $i\in I^{J,J'}$ and $j_i\in J$ is the non-zero entry as in (i), then $| a_{ij_i}|\ge c_1$ and ${\rm sign}(a_{ij_i}) = s_{j_i}$.

Denote 
$$
m=m(\kappa):=\kappa\sqrt{pn}\wedge\frac{1}{8p}.
$$

Then, there exist absolute constants $0<c_{\ref{comblemma}},c_{\ref{comblemma}}'<\infty$ depending only on $c_0,c_1$, such that

$$
\P \left(
\bigcap_{\kappa \le (8p\sqrt{pn})^{-1}\vee 1}
\bigcap_{s\in \{-1,1\}^{\kappa }}
\bigcap_{J\in \binom{[n]}{\kappa} }
\bigcap_{J'\in \binom{[n]}{m}, J\cap J'=\emptyset }
\A_{c_{\ref{comblemma}}'}^{J,J',s}
\right)\ge 1-\exp(-c_{\ref{comblemma}}pn).
$$
\end{lemma}
\begin{proof}
The proof is done by bounding the complement event. It is similar to Lemma 3.2 in \cite{InvofsparsenonHerm} but we need to take care of the sign and symmetricity.

Without loss of generality, we assume $c_1 = 1$ and only need to consider $s=(1,\cdots,1)$. For different choice of signs, the argument is identical. Fix $\kappa \le (8p\sqrt{pn})^{-1} \vee 1$ and  $J\in\binom{[n]}{\kappa}$, $J'\in \binom{[n]}{m}$. Let 
\begin{equation}
\begin{array}{rl}
I^1(J,J') := & \Big\{
i\in [n]\backslash(J\cup J'): a_{ij_i} \ge 1 {\rm \ for \ some\ }  j_i\in J,\\
& {\rm \ and \ }
a_{ij_i} = 0 {\rm \ for \ all\ }
j \in J\backslash{j_i}
\Big\}.   
\end{array}
\end{equation}

Similarly, we define
$$
I^0(J,J'):=\left\{
i\in [n]\backslash(J\cup J'): a_{ij}=0 {\rm \ for \ all\ } j\in J'
\right\}.
$$
Here we require $(I^1\cup I^0)\cap (J\cup J') = \emptyset$ so that we can get rid of symmetricity and achieve independence. On the other hand, since $m,\kappa \ll n$, this won't harm our probability bound.

To prove our desired result, we need to show the cardinality of $I^1(J,J')$ is at least $c\kappa pn$ with high probability for some constant $c$ firstly. Then we can apply Chernoff's inequality to prove that $|I'(J,J')\cap I^0(J,J')|$ is large with high probability. Finally, we take union bound over all different choices of $J,J',s$.

We start with obtaining a lower bound on $\P (i\in I^1(J,J'))$ for every $i\in [n]$. By our assumption on $a_{ij}$, we have for any $i\not\in J\cup J'$,
$$
\P(i\in I^1(J,J')) \ge c_0 |J| p(1-p)^{|J|-1} \ge c_0\kappa p (1-\kappa p)\ge \frac{c_0\kappa p}{2}.
$$
Therefore, by Chernoff's inequality and the fact that $\kappa, m \ll n$, we have 
$$
\P (|I^1(J,J')|\le \frac{c_0 \kappa p n}{4}) \le \exp(-c_1pn).
$$
Next, we fix a set $J'\in \binom{[n]}{m}$, for any $i\in [n]\backslash{(J\cup J')}$, we have that
$$
\P(i\in I^0(J,J'))= (1-p)^{|J'|} \ge 1- p|J'|=1-pm\ge \frac{3}{4}.
$$
Thus, for any given $I\subset[n]$, the random variable $I\backslash I^0(J,J')$ can be represented as the sum of independent Bernoulli variables taking value 1 with probability greater than $pm$. Also, note that 
$$
\E |I\backslash I^0(J,J')| \le pm |I| \le \frac{|I|}{4}
$$
by the assumption on $\kappa $ and $m$. Now, use Chernoff's inequality again, we have
$$
\P\left(|I\backslash I^0(J,J')| \ge \frac{|I|}{2}\right) \le \exp\left( -\frac{|I|}{16}\log\left(\frac{1}{4pm} \right) \right).
$$
So for any $I\subset[n]$ such that $|I|\ge \frac{c_0\kappa n p}{4}$, we can deduce that for any $J\in\binom{[n]}{\kappa}$,
\begin{equation}
\begin{array}{rl}
& \P\left(
\displaystyle\exists J'\in\binom{[n]}{m} {\rm \ such \ that \ } 
|I^0(J,J')\cap I| \le \frac{c_0\kappa pn }{8}
\right)\\
\le & \displaystyle\sum_{J'\in\binom{[n]}{m}} \P(|I\backslash I^0(J,J')| \ge \frac{1}{2}|I|)\\
\le & \displaystyle\binom{n}{m} \exp\left( -\frac{|I|}{16} \log\left( \frac{1}{4pm} \right) \right)\le \exp(-\kappa p n U),
\end{array}
\end{equation}
where
$$
U:= \frac{c_0}{64} \log \left( \frac{1}{4pm} \right)
-\frac{m}{\kappa pn} \log\left( \frac{en}{m} \right).
$$
Here we have $U\ge \frac{c_0}{100}$ (lower bound of $U$ is a direct computation which was done in proof of Lemma 3.2 in \cite{InvofsparsenonHerm} so we omit details here).
Now for any $J\in \binom{[n]}{\kappa}$, define
\begin{equation}
\begin{array}{rl}
p_J& :=\P \bigg(
 J'\in \binom{[n]}{m} {\rm \ such \ that \ } J'\cap J=\emptyset, \\
& |I'(J,J')\cap I^0(J,J')| <\frac{c_0 \kappa pn }{8}
\bigg).
\end{array}
\end{equation}

As $J,J'$ are disjoint, we have independence between random subsets $I^1(J,J')$ and $I^0(J,J')$. Thus
\begin{equation}
\begin{array}{rl}
     p_J & \le \displaystyle\sum_{I\subset [n],\ |I|\le \frac{c_0}{4}\kappa p n} \P(I^1(J,J') = I) \\
     & +   \displaystyle\sum_{I\subset [n],\ |I|> \frac{c_0}{4}\kappa p n} \P(I^1(J,J') = I)  
     \P \bigg( 
     \exists J'\in \binom{[n]}{m}\\
     & {\rm \ such\ that\ }|I^0(J,J')\cap I| \le \frac{c_0}{8}\kappa pn
     \bigg)\\
     & \le \exp(-c_1\kappa pn) + \exp(-c_2\kappa pn) \le \exp(-c_3\kappa pn).
\end{array}
\end{equation}

To finish the proof, we only need to take union bound over all different choices of $J$,  $s$ and $\kappa$. Set $c_{\ref{comblemma}}' = c_0 /8$. We have
$$
\P \left(
\bigcup_{s\in \{-1,1\}^{ \kappa }}
\bigcup_{J\in \binom{[n]}{\kappa} }
\bigcup_{J'\in \binom{[n]}{m}, J\cap J'=\emptyset }
\Big(\A_{c_{\ref{comblemma}}'}^{J,J',s}\Big)^c
\right)\le 2^{ \kappa}\binom{n}{\kappa}\exp(-c_3\kappa pn).
$$
Notice that the probability bound $\exp(-c_3\kappa p n)$ dominate $2^{\kappa}\binom{n}{\kappa}$ for $C$ large enough in $p\ge \frac{C\log n}{n}$, we have the above probability is bound by $\exp(-c_3\kappa pn/2)$. Finally take another union bound over $\kappa$ with finish our proof.

\end{proof}

Notice that to apply Lemma \ref{comblemma}, we need a two side tail probability estimate of a random variable with mean zero, variance 1 and bounded fourth moment. The following lemma although simple may have its own interest in some applications.

\begin{lemma}\label{tsbd}
Let $\xi $ be a random variable with mean zero, unit variance, and finite fourth moment $M_4^4$. Then there exist constant $c_{\ref{tsbd}},c'_{\ref{tsbd}}>0$ depending only on $M_4$ such that, 
$$
\min(\P(\xi \le -c_{\ref{tsbd}} ),\P(\xi \ge  c_{\ref{tsbd}} )) \ge c'_{\ref{tsbd}}.
$$
\end{lemma}

\begin{proof}
This lemma is a two-sided version of lemma 3.2 in \cite{SMstSvRect}. We derive a lower bound for second moment of positive and negative part separately and then use Paley-Zymmund inequality.

Let $\xi^+(t) = 1_{t>0}(t)\xi(t),\ \xi^-(t) = 1_{t<0}(t)\xi(t)$ be the positive and negative part of $\xi$. Suppose $\E(
\xi^+)^2 = a$. Then by Cauchy-Schwartz inequality, we have $\E \xi^+ \le a^{1/2}$. By $\E \xi =0$ and $\E \xi^2 = 1$, we have 
$$
\E |\xi^-| = \E \xi^+ \le a^{1/2}, \E (\xi^-)^2 = 1-a.
$$
Apply H\"{o}lder's inequality and $\E|\xi|^4= M_4^4$, we have
$$
1- a = \E (\xi^-)^2 = \E |\xi^-|^{2/3} |\xi^-|^{4/3}
\le (\E |\xi^-|)^{2/3} (\E |\xi^-|^4)^{1/3} \le a^{1/3} M_4^{4/3}.
$$
Thus $a$ is lower bounded by some constants $c$ depending only on $M_4$. 
Apply Paley-Zygmund inequality we have 
$$
\P \left(\xi^+ \ge \sqrt{\frac{c}{2}} \right)  
= \P \left(|\xi^+|^2 \ge  \frac{c}{2} \right) \ge \frac{(\E |\xi^+|^2 - c/2)^2}{M_4^4} \ge \frac{c^2}{4M_4^4}.
$$
The Lemma is proved by repeating the same argument for positive part.
\end{proof}

We now use the above Lemma \ref{comblemma} to establish a uniform small ball probability bound for the set of dominated vectors. Without loss of generality, we many assume that  $1/(8p)>1$. For $p\ge 1/8$, we only need to apply result on dense matrix (see \cite{Invofsym}) to prove our main theorem.

\begin{lemma}\label{domonly}
Consider $A$ satisfies \ref{assump} and $p\ge (1/4)n^{-1/3}$. For any $u\in \R^n$, there exist $c_{\ref{domonly}},c_{\ref{domonly}}',c_{\ref{domonly}}''$ depending only on $C_{op},M_4$, such that
\begin{equation}
\begin{array}{rl}
& \P\Big( \exists x \in \dom ( (8p)^{-1}, c_{\ref{domonly}}')
{\rm \ such \ that \ }\\
& \V Ax -u \V_2 \le c_{\ref{domonly}}'' \sqrt{np} {\rm \ and \ } \V A \V\le C_{op}\sqrt{pn} \Big)\\
\le & \exp(-c_{\ref{domonly}}pn).
\end{array}
\end{equation}

\end{lemma}

\begin{proof}[Proof of Lemma \ref{domonly}]
Our proof is similar to Lemma 3.3 in \cite{InvofsparsenonHerm}. The major difference is how to deal with the symmetricity.  We start with proving the result for $\sparse((8p)^{-1})$ vectors of unit length. Then we can prove that these estimates can be easily extended to the dominated vectors. The proof strategy for sparse vectors may depends on $p$ (see Lemma 3.3 in \cite{InvofsparsenonHerm}), but for our purpose, we only need to prove it for $p\ge (1/4)n^{-1/3}$.

Since $p\ge (1/4)n^{-1/3}$, we apply the combinatorial Lemma \ref{comblemma} with $\kappa=1$ and $m=\frac{1}{8p}$. Assuming that the event described in this lemma occurs, we split the vector into blocks with disjoint support. One of these blocks has a large $l_2-$norm. By Lemma \ref{comblemma}, a large number of rows of the matrix has only one non-zero entry in the columns corresponding to the support of this block. This will be sufficient for us to conclude that $\V A x -u \V_2$ is bounded from below for $x\in \sparse((8p)^{-1})$. Note that to get the small ball probability estimate, we also need $\min(\P(\xi \le -c ),\P(\xi \ge  c))  \ge c'$. This is guaranteed by Lemma \ref{tsbd}. 

With out loss of generality, we only need to work on  ${\rm sign}(u) = \{-1\}_{i=1}^n$. For general cases, we only need to work on $A' = -{\rm diag}({{\rm sign}(u)}) A$ and $u' = -{\rm diag}({{\rm sign}(u)}) u$ where $A'$ still satisfies condition of Lemma \ref{comblemma}.
For $k\in[n]$, set $J_k=\{k\}$ and $J_l'={\rm supp}(x) \backslash J_k$. Let $\A$ be the event that for each $k\in [n]$, $v \in\{-1,1\} $ there exists a set $I_k \subset [n]$ of rows such that $|I_k|=c_{\ref{comblemma}'} pn$, and for any $i\in I_k$, $a_{ik}v \ge c_{\ref{tsbd}} $ and $a_{ij}=0$ for $j\in {\rm supp}(x)\backslash {k}$, and {\rm supp}$(x)$ is non-intersect with $I_k$. The definition of the sets $I_k$ immediately implies that $I_k \cap I_{k'}=\emptyset$ for $k\neq k' \in {\rm supp}(x)$. By Lemma \ref{comblemma} and Lemma \ref{tsbd}, $\P (\A) \ge 1- \exp (-c_{\ref{comblemma}}pn)$ where $c_{\ref{comblemma}}$ depend only on $M_4$. This shows that condition on this large probability event $\A$, we have 
$$
\left\Vert
A x -u 
\right\Vert_2^2 
\ge 
\left\Vert
A x 
\right\Vert_2^2
\ge 
\sum_{k\in {\rm supp}(x)} \sum_{i\in I_k } |(Ax)_i|^2 \ge \sum_{k\in {\rm supp}(x)}
c_{\ref{comblemma}}'pn c_{\ref{tsbd}}^2 |x(k)|^2.
$$
Thus $\left\Vert A x -u \right\Vert_2 \ge c_1 \sqrt{pn}$ where $c_1$ depend only on $c_{op},M_4$. So we get the result proved for sparse vectors. This estimate can be automatically extended to the set $\dom\left( (8p)^{-1}, c_{\ref{domonly}}' \right)$, provided that the constant $c_{\ref{domonly}}'$ is small enough. Indeed, assume that 
\begin{equation}\label{Ac}
\begin{array}{rl}
\V Ax - u \V_2 < \frac{c_1}{2}\sqrt{ pn}
\end{array}
\end{equation}
for some $x\in \dom\left( (8p)^{-1}, c_{\ref{domonly}}' \right)$. Set $m=(8p)^{-1}$, it is easy to notice that $\V x_{[m+1:n]} \V_{\infty} \le m^{-1/2}$. Hence,
$$
\V x_{[m+1:n]}\V_2 \le c_{\ref{domonly}}' \sqrt{m} \V x_{[m+1:n]} \V_\infty \le c_{\ref{domonly}}',
$$
and therefore
\begin{equation}
\begin{array}{rl}
\V Ax_{[1:m]}\V_2 &\le \V Ax \V_2 +\V A \V \V x_{[m+1:n]} \V_2 \\
     & \displaystyle\le \frac{1}{2}\sqrt{c_1 pn} + C_{op} \sqrt{pn} c_{\ref{domonly}}' \le \frac{3}{4}\sqrt{c_1 pn}
\end{array}
\end{equation}
provide $c_{\ref{domonly}}'$ small enough. Furthermore, 
\begin{equation}
\begin{array}{rl}
\displaystyle\Big| \left\Vert A(x_{[1:m]}/\V x_{[1:m]} \V_2) \right\Vert_2
-\V Ax_{[1:m]}\V
\Big|     &\le C_{op} \left| 1- \V x_{[1:m]} \V_2 \right|  \\
     & \le \frac{1}{4}\sqrt{c_1 pn}.
\end{array}
\end{equation}

Since $x_{[1:m]}/\V x_{[1:m]} \V_2 \in \sparse((8p)^{-1})\cap S^{n-1}$, combining the above steps we note equality (\ref{Ac}) holds only in $\A^c$. Therefore, we proved the lemma with $c_{\ref{domonly}}=c_{\ref{comblemma}}$ and $c_{\ref{domonly}}''=c_1$.

\end{proof}

Similar to dominated vectors, we can extend the result of Lemma \ref{domonly} to compressible vectors. This step is simply an  approximation. Recall that $\sparse((8p)^{-1})\cap S^{n-1} \subset \dom ((8p)^{-1}, c)$ for any $c$.

\begin{lemma}\label{compinclu}
Consider $A$ satisfies \ref{assump} and $p\ge (1/4)n^{-1/3}$. For any $u\in \R^n$, there exist $c_{\ref{compinclu}},c_{\ref{compinclu}}',c_{\ref{compinclu}}''$ depending only on $C_{op},M_4$, such that
\begin{equation}
\begin{array}{rl}
& \P\Big( \exists x \in \comp ( (8p)^{-1}, c_{\ref{compinclu}}')
{\rm \ such \ that \ }\\
& \V Ax -u \V_2 \le c_{\ref{compinclu}}'' \sqrt{np} {\rm \ and \ } \V A \V\le C_{op}\sqrt{pn} \Big)\\
\le & \exp(-c_{\ref{compinclu}}pn).
\end{array}
\end{equation}

\end{lemma}

\begin{proof}
We first denote following set
\begin{equation}
\begin{array}{rl}
\Omega:= & \Big\{ 
\forall x\in \sparse(1/(8p))\cap S^{n-1},\ \V Ax - u\V_2 \ge c_{\ref{domonly}}\sqrt{pn} \\ 
& {\rm and }
\ \V A \V \le C_{op} \sqrt{pn}
\Big\}.
\end{array} 
\end{equation}
Then on $\Omega$, for any  $\bar{x}\in\comp((8p)^{-1},c_{\ref{compinclu}}')$, we can find $x\in \sparse(1/(8p))$ such that 
$$
\V Ax/\V x\V_2 -u \V_2 \ge c_{\ref{domonly}}\sqrt{ pn} {\rm \ and\ }\V x- \bar{x} \V_2 \le c_{\ref{compinclu}}'.
$$
This also implies $|1-\V x \V_2 |\le c_{\ref{compinclu}}'$. Therefore
\begin{equation}
\begin{array}{rl}
\V A \bar{x} -u \V_2  & \ge \V A x/\V x\V_2  -u \V_2 - \V A \V \displaystyle\left\Vert x- \frac{x}{\V x\V_2 } \right\Vert  - \V A\V \V x-\bar{x} \V_2 \\
     & \ge c_{\ref{compinclu}}'' \sqrt{pn}
\end{array}
\end{equation}
by choosing $c_{\ref{compinclu}}'$ small enough. Since by Lemma \ref{domonly}, $\P (\Omega)\ge 1-\exp(-c_{\ref{domonly}}pn)$, the result follows.
\end{proof}

\subsection{Vectors very close to moderately sparse}

Lemma \ref{domonly} provided uniform lower bound on $\V A x \V $ for vectors which are close to very sparse vectors. To prove Theorem \ref{mthcomp}, we need to uplift these estimates for vectors which are less sparse (see Section 3.2 in \cite{InvofsparsenonHerm}). These vectors are well spread ones which allows us to obtain a strong small ball probability estimate so that we can use the standard net argument. The argument is a modification of proof of Lemma 3.8 in \cite{InvofsparsenonHerm}.




As a direct application of Corollary 3.7 in \cite{InvofsparsenonHerm}, we have the following corollary.
\begin{corollary}\label{sbspread}
Let $A_n$ be an $n\times m$ matrix with i.i.d. entries of the form $a_{ij}=\xi_{ij}\delta_ij$, where $\xi_{ij},\delta_{ij}$ are the same as in Assumption \ref{assump}. Then for any $\alpha >1$, there exist $\beta,\gamma >0$, depending on $\alpha$ and the fourth moment of $\xi_{ij}$, such that for any $x\in \R^m$, satisfying $\V x\V_{\infty}/ \V x\V_2 \le \alpha \sqrt{p}$, we have 
$$
\L\left(
A_n x, \beta\sqrt{pn} \V x\V_2 \le \exp(-\gamma n)
\right).
$$
\end{corollary}

Applying these results on Levy concentration we now prove uniform lower bound on $\V A x \V_2$ for vectors in $\dom (M, c)$. Note that proof of following lemma is a direct modification of first part of Lemma 3.8 in \cite{InvofsparsenonHerm}. The only variation is we need to restrict on a block of $A$ to get the small ball probability estimate. 

\begin{lemma}\label{modspar}
Consider $A$ satisfies \ref{assump} and $p\ge (1/4)n^{-1/3}$. For any $u\in \R^n$ and $p^{-1} \le M \le c_{\ref{modspar}}'''n$, there exist $c_{\ref{modspar}},c_{\ref{modspar}}',c_{\ref{modspar}}''$ depending only on $C_{op},M_4$ such that, for any $u\in \R^n$, 
\begin{equation}
\begin{array}{rl}
& \P\Big( \exists x \in \dom \left( M, c_{\ref{modspar}}' \right)
{\rm \ such \ that \ }\\
& \V Ax -u \V_2 \le c_{\ref{modspar}}'' \sqrt{np} {\rm \ and \ } \V A \V\le C_{op}\sqrt{pn} \Big)\\
\le & \exp(-c_{\ref{modspar}}pn).
\end{array}
\end{equation}
\end{lemma}

\begin{proof}
For convenience, denote $m=(8p)^{-1}$, so we have $m<M/2$. Due to Lemma \ref{domonly} and \ref{compinclu}, it is enough to obtain a uniform lower bound for all vectors from the set
$$
W:=\dom\left( M, c_{\ref{modspar}}' \right) \backslash
\left(
\comp((8p)^{-1},c_{\ref{compinclu}}')\cup \dom((8p)^{-1}, c_{\ref{domonly}}')
\right).
$$
We start with a set with only $M$-sparse vectors
$$
V:=\sparse(M) \backslash
\left(
\comp((8p)^{-1},c_{\ref{compinclu}}')\cup \dom((8p)^{-1}, c_{\ref{domonly}}')
\right).
$$

Since $p\ge (1/4)n^{-1/3}$, the proof is based on the straightforward $\vv$-net argument as in Lemma 3.8 in \cite{InvofsparsenonHerm}. Since for any $x\in V, x\notin \left(
\comp((8p)^{-1},c_{\ref{compinclu}}')\cup \dom((8p)^{-1}, c_{\ref{domonly}}')
\right)$, we have that 
$$
\frac{\V x_{[m+1:M]} \V_{\infty}}{\V x_{[m+1:M]} \V_2} \le (c_{\ref{domonly}}')^{-1} \sqrt{8p}.
$$

Now for this given $x$, define $A^x$ to be the sub-matrix restricted on the columns corresponding to ${\rm supp}(x)$ and rows corresponding to $[n]\backslash {\rm supp}(x)$. Then $A^x$ is an $(n-M)\times M$ submatrix with i.i.d. entries. By Corollary \ref{sbspread} and properties of Levy's concentration function, we have
\begin{equation}
\begin{array}{rl}
& \L \left( Ax, c_1 \sqrt{pn} \V x_{m+1:M} \V_2  \right)      \\
\le& \L \left( A^x x, c_1 \sqrt{pn} \V x_{m+1:M} \V_2  \right)\\
\le&  \exp(-c_2n)
\end{array}
\end{equation}
where $c_1,c_2$ depending only on $C_{op},M_4$.

Now, we will use this estimate of the Levy concentration function to control the infimum over $V$.  Since $V\subset\sparse(M)$, note  that the set $V$ is contained in $S^{n-1}$ intersected with the union of coordinate subspaces of dimension $M$. Thus, for $\vv < c_{\ref{compinclu}}'c_{\ref{modspar}}'$, there exists an $\vv-$ net $\N\subset V$ of cardinality less than
$$
\binom{n}{M}\left( \frac{3}{\vv } \right)^M \le \exp\left(c_{\ref{modspar}}'''n \log \left( \frac{3e}{c_{\ref{modspar}}'''\vv} \right) \right).
$$
We used the assumption $M\le c_{\ref{modspar}}'''n$ in above estimate. Moreover, we can choose the constant $c_{\ref{modspar}}'''$ sufficiently small (depending on $\vv$) so that $|\N|\le \exp(c_2n/2)$.  Using the union bound argument, we have
$$
\P \left(
\exists x\in \N, u\in \R^n |
\V Ax- u \V_2 \le c_1 \sqrt{pn} \V x_{[m+1:M]} \V_2 
\right)\le \exp(-c_2n/2).
$$

Now we can approximate any point of $W$ by a point of $\N$. Assume that for any $x\in \N$, 
$$
\V Ax -u \V_2 \ge c_1 \sqrt{pn} \V x_{[m+1:M]} \V_2.
$$
Let $x'\in W$, then we can find $x\in \N$ such that 
$$
\V  x'_{[1:M]}/ \V x'_{[1:M]} \V_2  -x \V_2 \le \vv . 
$$
Now, we show that $x$ and $x'$ are close. Since $m\le M/2$ and all coordinates of $x'_{[M+1:n]}$ have smaller absolute value than those of $x'_{[1:M]}$, we have
$$
\sqrt{M} \V X'_{[M+1:n]} \V_{\infty} \le \sqrt{2} \V x'_{[m+1:M]} \V_2.
$$
Now recall that $x'\in \dom(M,c_{\ref{modspar}}')$, so we have 
$$
\V x'_{[M+1:n]} \V_2 \le c_{\ref{modspar}}' \sqrt{M} \V x'_{[M+1:n]} \V_{\infty}
\le \sqrt{2} c_{\ref{modspar}}' \V x'_{[m+1:M]} \V_2.
$$
Now, we can use the fact that $\V  x'_{[1:M]}/ \V x'_{[1:M]} \V_2  -x \V_2 \le \vv$ together with triangle inequality. Therefore,  we have
$$
\V x'_{[m+1:M]} \V_2 \le \V x'_{[1:M]} \V_2(\V x_{[m+1:M]} \V_2  + \vv)  \le \V x_{[m+1:M]} \V_2  + \vv.
$$
Now, for any $x\in \N$, $x\notin \comp{(m, c_{\ref{compinclu}}')}$, $\V x_{[m+1:M]} \V_2 \ge c_{\ref{compinclu}}' \ge  \vv$.  Applying previous two inequalities, we also have
\begin{equation}
\begin{array}{rl}
\V x'_{[M+1:n]} \V_2 & \le \sqrt{2} c_{\ref{modspar}}' \V x'_{[m+1:M]} \V_2 \\
& \le 2 c_{\ref{modspar}}' (\V x_{[m+1:M]} \V_2 +\vv)
\le 4 c_{\ref{modspar}}' \V x_{[m+1:M]} \V_2 
\end{array}
\end{equation}
and
\begin{equation}
\begin{array}{rl}
\V x- x'\V_2 & \le \left\Vert
x- x'_{[1:M]} / \V x'_{[1:M]} \V_2
\right\Vert_2
+\left| 1- \V x'_{[1:M]} \V_2 \right| + \V x'_{[M+1:n]} \V_2 \\
     & \le \vv + 2 \V x'_{[M+1:n]} \V_2
     \le \vv + 8 c_{\ref{modspar}}'\V x_{[m+1:M]} \V_2\\
     & \le 9 c_{\ref{modspar}}'\V x_{[m+1:M]} \V_2.
\end{array}
\end{equation}
Finally, by choosing $c_{\ref{modspar}}'$ sufficiently small, by the triangle inequality,
\begin{equation}
\begin{array}{rl}
\V Ax' -u \V_2 & \ge \V Ax -u \V - \V A \V \V x- x' \V_2   
 \\
     & \ge (c_1-9c_{\ref{modspar}}'C_{op})\sqrt{pn} \V x_{[m+1:M]}\V_2 \ge c_{\ref{modspar}}''\sqrt{pn}.
\end{array}
\end{equation}

\end{proof}

Now, we can conclude Theorem \ref{mthcomp}.
\begin{proof}
Theorem \ref{mthcomp} follow directly from a similar argument of Lemma \ref{compinclu}.

\end{proof}

Now, by Theorem \ref{mthcomp}, we have following small probability estimate similar to Proposition 4.2 in \cite{Invofsym}.
\begin{theorem}\label{sbcomp}{\rm (Small ball probability for compressible vectors)}. Consider $A$ satisfies \ref{assump} and $p\ge (1/4)n^{-1/3}$. For every $u\in \R^n$, one has
$$
\P\left( \inf_{\frac{x}{\V x\V_2}\in \comp(c_sn, c_d)} 
\V  Ax -u \V_2/\V x \V_2 \le c_{\ref{sbcomp}}'\sqrt{pn} \wedge \e_{op} 
\right)
\le 2 \exp(-c_{\ref{sbcomp}}pn)
$$
where $c_s,c_d,c_{\ref{sbcomp}},c_{\ref{sbcomp}}'$ depending only on $M_4,C_{op}$.

\end{theorem}

\begin{proof} 
Let $\e$ be the event in the left hand side whose probability need to be estimated. We start with some fixed small positive numbers of $c_s,c_d$ and $c_{\ref{sbcomp}}'$ which specific choice will be decided later. Conditioning on $\e$, we have that there exist vectors $u_0:=u/\V x \V_2 \in \span(u)$ and $x_0:=x/\V x \V_2 \in \comp(c_sn,c_d)$ such that 
$$
\V A x_0 -u_0 \V_2 \le c_{\ref{sbcomp}}'\sqrt{pn}.
$$
By definition of event $\e_{op}$, we have
$$
\V u_0 \V_2 \le \V Ax_0 \V + c_{\ref{sbcomp}}'\sqrt{pn} \le C_{op} \sqrt{pn} +c_{\ref{sbcomp}}'\sqrt{pn} \le 2C_{op} \sqrt{pn}
$$
Therefore
$$
u_0 \in \span(u)\cap 2 C_{op} \sqrt{pn} B_2^n=:E
$$
Let ${\mathcal M}$ be a $(c_1\sqrt{pn})$-net of the interval $E$ such that
$$
| {\mathcal M} | \le \frac{2C_{op} \sqrt{pn}}{c_1  \sqrt{pn}}= \frac{2C_{op}}{c_1}
$$
and choose $v_0 \in | {\mathcal M} |$ such that $\V v_0 -u_0 \V_2 \le c_1 \sqrt{pn}$. Then
$$
\V A x_0 -v_0 \V_2 \le c_{\ref{sbcomp}}' \sqrt{pn} +c_{1} \sqrt{pn}.
$$
Now we may choose $c_{\ref{sbcomp}}',c_1 \in (0,1)$ such that $c_{\ref{sbcomp}}'+c_1\le c_{\ref{mthcomp}}''$. So the event $\e$ implies the existence of vector $x\in \comp(c_sn,c_d),v_0\in\M$ such that $\V A x_0 -v_0 \V_2 \le c_{\ref{mthcomp}}''\sqrt{pn}$. Taking the union bound over $\M$, we have
$$
\P(\e) \le |\M| \max_{v_0\in \M} \P \left\{ 
\exists x \in \comp(c_sn,c_d) {\rm \ such \ that \ }
\V Ax - v_0 \V_2 \le c_{\ref{mthcomp}}'' \sqrt{np}
\right\}.
$$
Now we may apply Theorem \ref{mthcomp} together with the net  cardinalities estimates and we get
$$
\P(\e) \le \frac{2C_{op}}{c_1} \exp(-c_{\ref{sbcomp}}np).
$$
Use the condition on $p$ then we are done. The $c_s,c_d$ in this theorem can be chosen as $c_{\ref{mthcomp}}'$ and $c_{\ref{mthcomp}}'''$ in Theorem \ref{mthcomp}.

\end{proof}
\begin{remark}
{\rm Note that the constants $c_s,c_d$ can be chosen depending only on $C_{op},M_4$. These to constants are fixed in the later part of the proof. An immediate consequence of Theorem 
\ref{sbcomp} is }
\begin{equation}\label{lbforcomp}
\begin{array}{rl}
     \displaystyle\P \left\{
\inf_{x\in \comp(c_sn,c_d)} \V A x\V_2 \le \vv \sqrt{\frac{p}{n}} \wedge \e_{op}
\right\} & \le  2 \exp(-c_{\ref{sbcomp}}pn) \\
     & 
\end{array}
\end{equation}

\end{remark}

\section{Invertibility over incompressible vectors}\label{SecIoverInc}
Our goal in the following sections is to show, with high probability
$$
\min_{x\in \incomp{(c_sn,c_d)}} \V Ax \V_2 \gtrsim \sqrt{\frac{p}{n}}.
$$

\subsection{Incompressible vectors are spread}

Note that in Theorem \ref{sbcomp} and from now on, we will adapt the methodology of Vershynin in \cite{Invofsym} in order to decouple the symmetric matrix. Although some proofs are very similar to those in  \cite{Invofsym}, we still need to went through several proofs  in much detail under our setting. This need to be done to ensure the methodology works as well in sparsity setting. And what is more important is to  catch the affect of sparsity especially how it affect the probability bounds. For convenience of reader and to show the connection in methodology, we will try to use similar notation and structure as proofs in \cite{Invofsym}. 

First, we want to note that although the  incompressible vectors have many non-negligible coordinate but they have different advantage. Incompressible vectors $x$ have many coordinates that are well spread, that is to say a set of coordinates of size of order $n$ whose magnitudes are all of the  order $n^{-1/2}$. More precisely, we have the following lemma, see Lemma 3.4 in \cite{LOandInvofRM}:

\begin{lemma}\label{incompsp}{\rm (Incompressible vectors are spread)}. For every $x\in \incomp(c_0n,c_1)$, one has 
\begin{equation}\label{reg}
\begin{array}{rl}
\displaystyle\frac{c_1}{\sqrt{2n}}\le |x_k| \le \frac{1}{\sqrt{c_0 n}}
\end{array}
\end{equation}
for at least $\frac{1}{2}c_0 c_1^2 n$ coordinates $x_k$ of $x$.
\end{lemma}
We fix some constant $c_{oo}$ such that as in \cite{LOandInvofRM}
$$
\frac{1}{4}c_s c_d^2 \le c_{oo} \le \frac{1}{4}.
$$
Here note that  the value of $c_{oo}$ depend only on $c_s$ and $c_d$, which depend only on the parameters $C_{op}$ and $M_4$.  We may assign a subset  called spread$(x)\subset [n]$ for every vector $x\in\incomp(c_sn,c_d)$ such that 
$$
| {\rm spread}(x)| = \left \lceil{c_{oo}n}\right \rceil
$$
and the property in Lemma \ref{incompsp} hold for any $k\in {\rm spread}(x)$. The point here is that not all of the coordinates $x_k$ satisfying Lemma \ref{incompsp} will be good, the set ${\rm spread}(x)$ will allow us to only focus on the good coordinates. At this point, we may consider an arbitrary valid assignment of ${\rm spread}(x)$ to $x$, the particular choice will be decide later in the proof.

\subsection{Distance problem via small ball probabilities for quadratic forms}
To derive incompressible part of the invertibility problem, we need the following Lemma, see Lemma 2.4 in \cite{InvofsparsenonHerm}.
\begin{lemma}\label{invviadist} {\rm (Invertibility via distance)}. For $j\in [n]$, let $A_j$ denote the $j-$th column of $A_n$, and let $H_j$ be the subspace of $\R^n$ spanned by ${A_i,i\in [n]\backslash j}$. Then for any $\vv, \rho>0$, and $M<n$, we have 
$$
\P \left( \inf_{x\in \incomp (M, \rho)} \V A x \V_2 \le \vv \sqrt{\frac{p}{n}} \right) \le \frac{1}{M} \sum_{j=1}^n
\P\left(
\dist (A_j, H_j) \le  \sqrt{p} \vv
\right)
$$
\end{lemma}

So we may reduces the invertibility problem to the distance problem, namely an upper bound on the probability 
$$
\P\left(
\dist (A_1, H_1) \le c_1 \sqrt{p} \vv
\right)
$$
where $A_1$ is the first column of $A$ and $H_1$ is the span of the other column. (By a permutation of the indices in $[n]$, the same bound would hold for all $\dist(A_k,H_k)$ as required in Lemma \ref{invviadist}).

But we have a symmetric matrix, to do the decoupling we  need tools to evaluate the distance problem. To this end, the following proposition in \cite{Invofsym} reduces the distance problem to the small ball probability for quadratic forms of random variables:
\begin{proposition}\label{distqua}{\rm (Distance problems via quadratic forms)}. Let $A=(a_{ij})$ be an arbitrary $n\times n$ matrix. Let $A_1$ denote the first column of $A$ and $H_1$ denote the span of the other columns. Furthermore, let $B$ denote the $(n-1)\times(n-1)$ minor of $A$ obtained by removing the first row and the first column from $A$, and let $X\in \R^{n-1}$ denote the first column of $A$ with the first entry removed. Then 
$$
\dist (A_1, H_1) =\frac{| \langle B^{-1} X, X\rangle -a_{11}|}{\sqrt{
1+\V B^{-1} X\V_2^2}}.
$$

\end{proposition}

\begin{remark}
{\rm We may apply Proposition \ref{distqua} to the $n\times n$ random matrix $A$ which we studied. Consider $a_{1,1}$ as an arbitrary fixed number and bound our probability uniformly for all $a_{1,1}$, the problem reduces to estimating the small ball probability for the quadratic form $\langle B^{-1} X, X\rangle$. The random matrix $B$ has the same structure as $A$ except for the dimension is $n-1$. Thus it will be convenient to develop the theory in dimension $n$ for the quadratic forms $\langle A^{-1} X, X \rangle $, where $X$ is an independent random vector (see Remark 5.2 in \cite{Invofsym}).  }
\end{remark}

\subsection{Small ball probabilities for quadratic forms via additive structure}
It is a popular and powerful to estimate small ball probabilities using the additive structure of vectors. For completion of our argument, let us first review the the Littlewood-Offord theory and its extension to quadratic forms by decoupling, see \cite{Invofsym}.

Linear Littlewood-Offord theory concerns the small ball probabilities for the sums of the form $S=\sum x_k \xi_k$ where $\xi_k$ are identically distributed independent random variables, and $x=(x_1,\cdots,x_n)\in S^{n-1}$ is a given coefficient vector. The   additive structure of $x\in \R^n$ is characterized by the least common denominator (LCD) of $x$. If the coordinates $x_k=p_k/q_k$ are rational numbers, one can  measure the additive structure in $x$ using the least denominator $D(x)$ of these ratios, which is the common multiple of the integers $q_k$. In the other words, $D(x)$ is the smallest number $\theta>0$  such that $\theta x \in \Z^n$. An extension of this concept for general vectors with real coefficients was developed in \cite{LOandInvofRM,SMstSvRect, Invofsym} which give us following definition of LCD.

\begin{definition}\label{LCDdef1} {\rm (Least Common Denominator)}. Let $L\ge 1$. We defined the least common denominator (LCD) of $x\in S^{n-1}$ as 
$$
D_{L}(x) = \inf \left\{
\theta>0: \dist(\theta x, \Z^n) < L \sqrt{\log_{+}(\theta /L)}
\right\}.
$$
\end{definition}
\begin{remark}
{\rm If the vector $x$ is considered in $\R^I$ for some subset $I\subset [n]$, then in this definition we replace $\Z^n$ by $\Z^I$.}
\end{remark}
It can be easily seen that we always have  $D_L(x)>L$. We may also notice that the parameter $L$ is up to our choice.  Recall by Remark \ref{delta0} that there exists $\delta_0,\vv'_0\in (0,1)$, such that for any $\vv <\vv'_0$, $\L (\xi_{ij}\delta_{ij},\vv)\le 1- \delta_0 p$. Due to the sparsity, we will often use the parametrization $L=(\delta p)^{-1/2}$ in our proofs (also see Section 4 of \cite{InvofsparsenonHerm}).
\begin{remark}
{\rm We may refer $D_L(x)$ as $D(x)$ for convenience.}
\end{remark}

Another useful bound is the following, see Lemma 6.2 in \cite{Invofsym}.
\begin{lemma}
For every $x\in S^{n-1}$ and every $L\ge 1$, one has
$$
D_L(x)\ge \frac{1}{\V x\V_{\infty}} 
$$
\end{lemma}

Now we can try to express the small ball probabilities of sums $\L (S,\vv)$ in terms of $D_L(x)$. This was done in the following theorem, see Theorem 6.3 in \cite{Invofsym}. 

\begin{theorem}\label{sblcd}{\rm (Small ball probabilities via LCD)}. Let $\xi_1,\cdots,\xi_n$ be independent and identically distributed random variables. Assume that there exist numbers $\vv_0, p_0,M_1 >0$ such that $\L (\xi_k,\vv_0)\le 1- p_0$ and $\E |\xi_k| \le M_1$ for all $k$. Then there exists $C_{6,3}$ which depends only on $\vv_0$, $p_0$ and $M_1$, and such that the following holds. Let $x\in S^{n-1}$ and consider the sum $S=\sum_{k=1}^{n} x_k \xi_{k}$. Then for every $L\ge p_0^{-1/2}$ and $\vv \ge 0$ one has
$$
\L(S,\vv) \le C_{\ref{sblcd}} L \left( \vv + \frac{1}{D_L(x)}  \right)
$$
for some constant $C_{\ref{sblcd}}$ depending only on second and fourth moments of $\xi$.
\end{theorem}

Applying the above theorem to the sparse vector, one may get following theorem for sparse vector, see Proposition 4.2 in \cite{InvofsparsenonHerm}.
\begin{theorem}\label{sblcd2}{\rm (Small ball probabilities via LCD)}. Let $S\in \R^n$ be a random vector with i.i.d. coordinates of the form $S_j=\delta_j \xi_j$, where $\P (\delta_j =1) =p$, and $\xi_j$s are random variables with unit variance, and finite fourth moment, which are independent of $\delta_j$. Then for any $v\in S^{n-1}$, $L=(\delta p)^{-1/2}$ and $\delta < \delta_0$
$$
\L\Big(\sum_{j=1}^n S_j v_j, \sqrt{p} \vv \Big) \le C_{\ref{sblcd2}}  \left( \vv + \frac{1}{\sqrt{p}D_{L}(v)}  \right)
$$
for some constant $C_{\ref{sblcd2}}, \delta_0$ depending only on fourth moments of $\xi_j$.
\end{theorem}

\subsection{Regularized LCD}
As we discussed, the distance problem reduces to a quadratic Littlewood-Offord problem. Similar to \cite{Invofsym}, we want to the use the same technique to reduce the quadratic problem to a linear one by decoupling and conditioning arguments. This process requires a more robust version of the concept of the LCD, which R. Vershynin developed in \cite{Invofsym}. 

\begin{definition}\label{rLCDdef1}{\rm (Regularized LCD)}. Let $\lambda \in (0,c_{oo})$ and $L\ge 1$. We define the regularized LCD of a vector $x\in \incomp(c_sn,c_d)$ as
$$
\hat{D}_L(x,\lambda) = \max \left\{
D_L(x_I / \V x_I \V_2) : I \subset{\rm spread}(x), |I| = \lambda n
\right\}.
$$
Denote by $I(x)$ the maximizing set $I$ in this definition
\end{definition}
\begin{remark}{\rm 
Since the sets $I$ in this definition are subsets of ${\rm spread}(x)$, inequality are subsets of ${\rm spread}(x)$, inequalities (\ref{reg}) imply that }
$$
c \sqrt{\lambda} \le \V x_{I} \V_2 \le C \sqrt{\lambda}
$$
{\rm where $c=c_d/\sqrt{2}$ and $C=1/\sqrt{c_s }$.}
\end{remark}

We also have the following estimate for regularized LCD, see Lemma 6.8 in \cite{Invofsym}. 

\begin{lemma}\label{lbrlcd}
For every $x\in \incomp(c_sn,c_d)$ and every $\lambda \in (0,c_{oo})$ and $L\ge 1$, one has 
$$
\hat{D}_L(x,\lambda) \ge c_{\ref{lbrlcd}} \sqrt{\lambda n}
$$
where $ c_{\ref{lbrlcd}}$ depends only on $c_s$ and $c_d$.
\end{lemma}

We now state a version of Theorem \ref{sblcd} for regularized LCD, see Proposition 6.9 in \cite{Invofsym}.

\begin{theorem}\label{sbrlcd}{\rm (Small ball probabilities via regularized LCD)}. Let $\xi_1,\cdots,\xi_n$ be independent and identically distributed random variables. Assume that there exist numbers $\vv_0, p_0,M_1 >0$ such that $\L (\xi_k,\vv_0)\le 1- p_0$ and $\E |\xi_k| \le M_1$ for all $k$. Then there exists $C_{\ref{sbrlcd}}$ which depends only on $\vv_0$, $p_0$ and $M_1$, and such that the following holds. 

Consider a vector $x\in \incomp{(c_sn,c_d)}$ and a subset $J\subseteq [n]$ such that $J\supseteq I(x)$. Consider also $S_J=\sum_{k\in J} x_k \xi_{k}$. Then for every $\lambda \in (0,c_{oo})$ and $L\ge p_0^{-1/2}$ and $\vv \ge 0$, one has
$$
\L(S_J,\vv) \le C_{\ref{sbrlcd}} L \left( \frac{\vv}{\sqrt{\lambda}} + \frac{1}{\hat{D}_L(x,\lambda)}  \right).
$$
\end{theorem}

Similarly, we can rewrite it for sparse random sums. 
\begin{theorem}\label{sbrlcd2}
Let $S\in \R^n$ be a random vector with i.i.d. coordinates of the form $S_j=\delta_j \xi_j$, where $\P (\delta_j =1) =p$, and $\xi_j$s are random variables with unit variance, and finite fourth moment, which are independent of $\delta_j$. Consider a vector $x\in \incomp{(c_sn,c_d)}$ and a subset $J\subseteq [n]$ such that $J\supseteq I(x)$.  Then for every $\lambda \in (0,c_{oo})$, $v\in S^{n-1}$, $L=(\delta p)^{-1/2}$ and $\delta < \delta_0$
$$
\L\Big(\sum_{j=1}^n S_j v_j, \sqrt{p} \vv \Big) \le C_{\ref{sbrlcd2}} \left( \frac{\vv}{\sqrt{\lambda}} + \frac{1}{\sqrt{p}\hat{D}_{L}(x,\lambda)}  \right)
$$
for some constant $C_{\ref{sbrlcd2}}, \delta_0$ depending only on fourth moments of $\xi_j$.
\end{theorem}

By Theorem \ref{tens}, one has the following proposition as a corollary, see Proposition 6.11 in \cite{Invofsym}:
\begin{proposition}\label{sbasrlcd}{\rm  (Small ball probabilities for $Ax$ via regularized LCD.)} Let $A$ be a random symmetric matrix with mean zero variance one and fourth moment $M_4^4$ i.i.d. entries above diagonal. Let $x\in \incomp(c_sn,c_d)$ and $\lambda\in (0,c_{oo})$. Then for every $L\ge L_0$ and $\vv\ge 0$, one has
$$
\L (Ax,\vv \sqrt{n}) \le \left[
\frac{C_{\ref{sbasrlcd}}L\vv}{\sqrt{\lambda }} + \frac{C_{\ref{sbasrlcd}}L}{\hat{D}_L(x,\lambda)}
\right]^{n-\lambda n}.
$$

Here $C_{\ref{sbasrlcd}}$ and $L_0$ depend only on the parameters $M_4$.
\end{proposition}

It can be easily derived as a corollary that for $A$ is a sparse matrix, we have the following result:
\begin{proposition}\label{sbasrlcd2}{\rm (Small ball probabilities for $Ax$ via regularized LCD where $A$ is sparse.)} Let $A$ be a random matrix satisfies Assumption \ref{assump}. Let $x\in \incomp(c_sn,c_d)$ and $\lambda\in (0,c_{oo})$. Then one has for $L=(\delta p)^{-1/2}$ and $\delta<\delta_0$
$$
\L (Ax,\vv \sqrt{pn}) \le \left[ 
\frac{C_{\ref{sbasrlcd2}}\vv}{\sqrt{\lambda }} + \frac{C_{\ref{sbasrlcd2}}}{\sqrt{p}\hat{D}_{L}(x,\lambda)}
\right]^{n-\lambda n}.
$$

Here $C_{\ref{sbasrlcd}}, \delta_0$  depends only on the parameters $M_4$.
\end{proposition}

\section{Estimating additive structure}\label{SecStruc}
To estimate the small ball probability for quadratic form $\langle A^{-1} X, X \rangle$, we will first need to estimate the additive structure in the random vector $A^{-1}X$. In this section, we will show that the regularized LCD of $A^{-1}X$ is large for every fixed $X$ which is an analog of Theorem 7.1 in \cite{Invofsym} for sparse matrices.

\begin{theorem}\label{stsp}{\rm (Structure theorem for sparse matrix.)}
Let $A$ be a random matrix which satisfies Assumption \ref{assump} and $p \ge n^{-c_p}$.  Let $u\in \R^n$ be an arbitrary fixed vector, and consider $x_0:=A^{-1}u/\V A^{-1}u \V_2$. Let $n^{c_{\ref{stsp}}n/6}p^{-1/2}\ge L=(p\delta)^{-1/2} \ge (p\delta_0)^{-1/2}$, $p\ge n^{-c_p}$ and $n^{-c_{\ref{stsp}}}\le \lambda \le c_{\ref{stsp}}/4$. Consider the event
$$
\e =\left\{ x_0 \in \incomp(c_sn,c_d){\rm \ and \ } \hat{D}_L(x_0,\lambda)
\ge L^{-2} n^{c_{\ref{stsp}}/\lambda} \right\}
$$
Then
$$
\P (\e^c\cap \e_{op}) \le 2 e^{-c_{\ref{stsp}}'pn}.
$$
Here $c_p, c_{\ref{stsp}},c_{\ref{stsp}}', \delta_0 >0$ depend only on the parameters $C_{op}$ and $M_4$.
\end{theorem}
\begin{remark}
{\rm Theorem \ref{stsp} is the step that $p\ge n^{-c_p}$ is needed. To improve Theorem \ref{mth}, one just need to improve Theorem \ref{stsp} to work for a greater range of $p$. }
\end{remark}

We shall first prove the easier part that $x_0\in \incomp(c_s n,c_d)$ w.h.p.. The more difficult part of the theorem is the estimate on the LCD.

\begin{lemma}\label{proincomp}{\rm ($A^{-1}u$ is incompressible.) }
In the setting of Theorem \ref{stsp}, consider the event 
$$
\e_{1}=\left\{ x_0 \in \incomp(c_s n,c_d) \right\}
$$
Then
$$
\P (\e_1^c \cap \e_{op}) \le 2 \exp(-c_{\ref{proincomp}}pn)
$$
Here  $c_{\ref{proincomp}}$ depends only on the parameters $C_{op}$ and $M_4$.
\end{lemma}
\begin{proof}
Denote $x=A^{-1}u$, then $Ax=u$. Hence
$$
\e_1^c \subseteq \left\{ \exists x\in \R^n : \frac{x}{\V x\V_2}
\in \comp(c_sn,c_d) \wedge Ax=u
\right\}
$$
By Proposition \ref{sbcomp}, $\P (\e_1^c \cap \e_{op}) \le 2\exp(-c_{\ref{sbcomp}}np) $ .
\end{proof}

Following the strategy in \cite{Invofsym}, to get the structure theorem, we also need a special entropy estimate. This is done in Proposition 7.4 of \cite{Invofsym}. To state the result, we need the following definition first.

\begin{definition}\label{slLCD}{\rm (Sublevel sets of LCD)}.
Let us fix $\lambda \in (0,c_{oo})$. For every value $D\ge 1$, we define the set
$$
S_D=\left\{ x\in \incomp(c_s n,c_d): \hat{D}_L(x,\lambda )\le D \right\}
$$
\end{definition}

Then recall following covering Lemma, see Proposition 7.4 in \cite{Invofsym}.

\begin{lemma}\label{cvsublevel}{\rm (Covering sublevelsets of regularized LCD)}. Let $\lambda \in (C_{\ref{cvsublevel}}/n,c_{oo})$ and $L\ge 1$. For every $D\ge 1$, the sublevel set $S_D$ has a $\beta -$net $\N$ such that 
$$
\beta = \frac{L\sqrt{\log D}}{\sqrt{\lambda }D },\ 
|\N| \le \left[ \frac{C_{\ref{cvsublevel}}D}{(\lambda n)^{c_{\ref{cvsublevel}}}} \right]^n D^{1/\lambda }
$$
where $C_{\ref{cvsublevel}},c_{\ref{cvsublevel}}$ depend only on $c_s,c_d$. More precisely, $c_{\ref{cvsublevel}}=c_{oo}/4$.
\end{lemma}


\begin{remark}
{\rm The dominating term in the net size is the term $(\lambda n)^c$. However, once we adapt this cardinality estimate in the sparse case, the $(\lambda n)^{-cn}$ term need to dominate $p^n$, this end up with a limitation of the sparsity level $p$ in our proof.} 
\end{remark}

In Proposition \ref{sbcomp}, we estimated the small ball probabilities for the random vector $Ax$ for a fixed vector $x$. Now we combine it with Lemma \ref{cvsublevel} to obtain a bound that is uniform over all $x$ with small regularized LCD.  

\begin{lemma}\label{sbsublevel} {\rm (Small ball probabilities on a sublevel set of LCD.)} There exist $ \delta_0, c_{\ref{sbsublevel}},c_{\ref{sbsublevel}}, c_p$  depend only on $C_{op}$ and $M_4$, and such that the following hold. Let $n^{c_{\ref{sbsublevel}}n/6}p^{-1/2}\ge L=(p\delta)^{-1/2} \ge (p\delta_0)^{-1/2}$, $n^{-c_{\ref{sbsublevel}}}\le \lambda \le c_{\ref{sbsublevel}}/4$, $p\ge n^{-c_p}$ and $1\le D\le (L)^{-2}n^{c_{\ref{sbsublevel}}/\lambda}$. Then 
$$
\P \Big\{ \exists x\in S_D: \V Ax -u \V_2 \le C_{op}\beta \sqrt{pn} \wedge \mathcal{E}_{op} \Big\} \le n^{-c_{\ref{sbsublevel}}'n}
$$
where 
$$
\beta = \frac{L\sqrt{\log(2D)}}{\sqrt{\lambda }D}.
$$

\end{lemma}

\begin{proof} 
In this proof, the sparsity would play an important role. Unlike the non-sparse case in proof of Lemma 7.9 in \cite{Invofsym}. This proof would only work when $p$ is relatively large. And this is the reason we have to force some assumption for our main theorem of the paper.

We start with estimating the probability for $S_D/S_{D/2}$ instead of $S_D$. Proposition \ref{sbasrlcd2} implies that for every $s\in S_D\backslash S_{D/2}$, 
$$
\P \left\{ \V Ax - u\V_2 \le \vv \sqrt{pn} \right\}
\le \left[
\frac{C_{\ref{sbasrlcd2}}\vv }{\sqrt{\lambda }} + \frac{C_{\ref{sbasrlcd2}}}{\sqrt{p}D}
\right]^{n-\lambda n}, \ \vv \ge 0.
$$
Now we apply this for $\vv =2C_{op} \beta$. Since $\frac{\vv }{\sqrt{\lambda}}$ dominates  $\frac{1}{\sqrt{p}D}$, we have
$$
\P \left\{ \V Ax - u\V_2 \le 2C_{op} 
\beta \sqrt{pn} \right\}
\le  \left[
\frac{C L\sqrt{\log (2D)}}{\lambda D} 
\right]^{n-\lambda n}=:p_0
$$
where $C$ depend only on $M_4,C_{op}$. Now, choose a $\beta-$ net $\N$ of $S_D\backslash S_{D/2}$ according to Lemma \ref{cvsublevel}. We have

\begin{equation}
\begin{array}{rl}
     &  \P \left\{ \exists x \in \N: \V Ax -u \V_2 \le C_{op} \beta \sqrt{n}     \right\} \le |\N| p_0\\
     \le & \displaystyle\left[ \frac{C_{\ref{cvsublevel}}D}{(\lambda n)^{c_{\ref{cvsublevel}}}} \right]^n D^{1/\lambda} \left[
\displaystyle\frac{C L \sqrt{\log (2D)}}{\lambda D} 
\right]^{n-\lambda n}=:p_1.
\end{array}
\end{equation}
To estimate $p_1$, notice that  $n$ is sufficiently large, $n^{-c}\le \lambda \le c_{\ref{sbsublevel}}/4$ and $1\le D \le L^{-2} n^{c/\lambda}$. By choosing $c$ small enough, we have
\begin{equation}
\begin{array}{rl}
    p_1 & \le  \displaystyle C^n D^{\lambda n+ 1/\lambda}
    (\lambda n)^{-c_{\ref{cvsublevel}}n} L^n \lambda^{-n}
    ( \sqrt{\log (2D)} )^{n} \\
     & \le  \displaystyle C^n n^{2c n+ 1/\lambda^2}
     n^{-c_{\ref{cvsublevel}}n/2} L^n \lambda^{-n}
    ( c\log n /\lambda )^{n} \\
     & \le  \displaystyle 
     n^{-c_{\ref{cvsublevel}}n/3} L^n .    \\
\end{array}
\end{equation}
Choosing the constant $c_p$ sufficient small and we obtain 
$$
p_1\le n^{-c' n}
$$
where $c'$ depend only on $M_4,C_{op}$.
Assume event $\e_{op}$ hold and there exists $x\in S_D\backslash S_{D/2}$ such that $\V Ax -u \V_2 \le C_{op}\beta \sqrt{n}$. Then there exists $x_0\in \N$ such that $\V x - x_0 \V_2 \le \beta $. Therefore

\begin{equation}
\begin{array}{rl}
\V A x_0 -u \V_2 & \le \V Ax -u \V_2 +\V A(x-x_0) \V_2  \le 
\V Ax -u\V_2 +\V A \V \V x -x_0 \V_2\\
     & \le 2C_{op} \beta \sqrt{pn}.
\end{array}
\end{equation}
The probability of the later event is bounded by $p_1 \le n^{-c'n}$. So we have 
$$
\P \left\{
\exists x\in S_D\backslash S_{D/2}: \V Ax-u\V_2 \le C_{op} \beta \sqrt{pn} \wedge \e_{op}
\right\}\le n^{-c'n}.
$$
To remove $S_{D/2}$ in this bound, we divide it into level sets. Since $\beta$ decreases in $D$, the previous result can be applied for $D/2$ instead of $D$ if $D\ge 2$. Therefore
$$
\P \left\{
\exists x\in S_{D/2}\backslash S_{D/4}: \V Ax-u\V_2 \le C_{op} \beta \sqrt{pn} \wedge \e_{op}
\right\}\le n^{-c'n}.
$$
We can continue defining such sets for $S_{D/4}\backslash S_{D/8}$ and so on. On the other hand, $S = \bigcup_{k=0}^{k_0} (S_{2^{-k}D})$, where $k_0$ is the largest integer such that $2^{-k_0}D \ge c_{\ref{lbrlcd}}\sqrt{\lambda n}$. By Proposition \ref{lbrlcd},  $S_{D_0}$ is empty set if $D_0 < c_{\ref{lbrlcd}}\sqrt{\lambda n}$. Since $c_{\ref{lbrlcd}}\sqrt{\lambda n} \ge 1$, we have $k_0 \le \log_2 (D)$. Therefore
$$
\P \left\{
\exists x \in S_D: \V Ax - u\V_2 \le K\beta \sqrt{pn} \wedge \e_{op}
\right\}
 \le \log_2 (D) n^{-c'n} \le n^{c''n}
$$
if the constant $c''$ is chosen appropriately small. 

\end{proof}

\begin{proof}[Proof of Theorem \ref{stsp}]
This is a direct analog of proof of Theorem 7.1 in \cite{Invofsym}. We now fix constants $ \delta_0, c_{\ref{sbsublevel}},c_{\ref{sbsublevel}}, c_p$ in Lemma \ref{sbsublevel}. Define
$$
\e_0 =\left\{
\hat{D}_L(x_0,\lambda) > L^{-2} n^{c_{\ref{sbsublevel}}/\lambda}=:D_0
{\rm \ or \ } \hat{D}_L (x_0,\lambda) {\rm \ is \ undefined}
\right\}
$$
and
$$
\e_1=\left\{
x_0\in\incomp(c_sn,c_d)
\right\}.
$$
Note $\hat{D}_L(x_0,\lambda )$ is defined if $\e_1$ holds. Thus we may rewrite $\e$ as 
$$
\e = \e_1 \cap \e_0.
$$
Then 
$$
\e^c= \e_1^c \cup (\e_1\cap \e^c)=\e_1^c \cup (\e_1 \cap \e_0^c).
$$ 
So the probability we want to estimate can be rephrased as
$$
\e^c \cap \e_K \subseteq (\e_1^c \cap \e_K)\cup 
(\e_1 \cap \e_0^c \cap \e_K).
$$
Thus
$$
\P(\e^c\cap \e_K) \le \P (\e_1^c \cap \e_K)
+\P(\e_1 \cap \e_0^c\cap \e_K).
$$
By Lemma \ref{proincomp}, the first term can be bounded to be:
$$
 \P (\e_1^c \cap \e_K) \le 2\exp(-c_{\ref{proincomp}}pn).
$$
To estimate the second term $\P(\e_1 \cap \e_0^c \cap \e_K)$, consider
$$
\e_1 \cap \e_0^c \cap \e_K
=\left\{
x_0 := A^{-1} u/ \V A^{-1} \V_2 \in S_{D_0} \wedge \e_K
\right\}.
$$
Define $u_0:= Ax_0 = u/ \V A^{-1} u \V_2$ and $\e_K$ implies 
$$
\V u_0 \V_2 = \V Ax_0\V_2 \le \V A\V \le C_{op}\sqrt{pn}.
$$ 
Thus, $u_0$ belongs to  a one-dimensional interval. More precisely, 
$$
u_0\in \span(u)\cap C_{op}\sqrt{pn} B_2^n =:E.
$$
So
$$
\e_1\cap \e_0^c \cap \e_K \subseteq 
\left\{
\exists x_0 \in S_{D_0}, \exists u_0 \in E : 
Ax_0 =u_0 \wedge \e_{op}
\right\}.
$$
Now, choose 
$$
\beta_0 = \frac{L\sqrt{\log (2D_0)}}{D_0}.
$$
Let $\M$ be some fixed $(C_{op}\beta_0 \sqrt{pn})-$net of the interval $E$ with cardinality
$$
|\M| \le \frac{C_{op}\sqrt{pn}}{C_{op}\beta_0 \sqrt{pn}}
=\frac{1}{\beta_0}\le D_0.
$$
Therefore for $u_0 \in E$ we there exists $v_0 \in \M$ such that $\V u_0 - v_0 \V_2 \le C_{op}\beta_0 \sqrt{pn}$. We also have $\V Ax_0 -v_0 \V_2 \le C_{op}\beta \sqrt{np}$ since $Ax_0 = u_0$. Therefore
$$
\e_1 \cap \e_0^c \cap \e_K
\subseteq \left\{
\exists x_0 \in S_{D_0}, \exists v_0 \in \M:
\V Ax_0 -v_0 \V_2 \le C_{op}\beta_0 \sqrt{pn} \wedge \e_{op}
\right\}.
$$
Finally, applying Lemma \ref{sbsublevel} and a union bound argument for all $v_0\in \M$,
$$
\P (\e_1 \cap \e_0^c\cap  \e_{op}) \le |\M|n^{-c_{\ref{sbsublevel}}'n} \le D_0 n^{-c_{\ref{sbsublevel}}'n} \le n^{-c_{\ref{sbsublevel}}'n/2} 
$$
where  $D_0 \le n^{c/\lambda}$, and since we can assume that constant  $c_{\ref{sbsublevel}}>0$ sufficient small. Our proof is complete.

\end{proof}

\section{Small ball probability for quadratic forms}\label{SecSBquad}

Now, we use the machinery developed in \cite{Invofsym} to estimate small ball probabilities. Recall that by Proposition \ref{distqua}, the distance problem reduces to estimating Levy concentration function for the self-normalized quadratic forms:
\begin{equation}\label{lvqua}
\L \left\{ 
\frac{ |\langle A^{-1}X,X \rangle |}{\sqrt{1+\V A^{-1} X\V_2^2}}, \vv \sqrt{p}
\right\}.
\end{equation}

The goal of this section is to prove the following estimate, for the non-sparse version, see Theorem 8.1 in \cite{Invofsym}.

\begin{theorem}\label{sbqua}{\rm (Small ball probabilities for quadratic forms.)} Let $A$ be an $n\times n$ random matrix satisfies Assumption \ref{assump} and $p\ge n^{-c_p}$. Let $X$ be a random vector in $\R^n$ whose entries are 
identically distributed, and satisfy the same assumption as those of $A$. There exist constants $c_p, C_{\ref{sbqua}},c_{\ref{sbqua}},c_{\ref{sbqua}}'$ depend only on the parameters $C_{op}$ and $M_4$, and such that the following holds. For every $\vv\ge 0$ and $u\in \R$, one has
$$
\P \left\{
\frac{|\langle A^{-1}X,X \rangle - u|}{\sqrt{1+\V A^{-1} X\V_2^2}} \le \vv\sqrt{p} \wedge \e_{op} 
\right\}
\le C_{\ref{sbqua}}\vv^{1/9} +2\exp(-n^{c_{\ref{sbqua}}})+\exp(-c_{\ref{sbqua}}'pn).
$$
\end{theorem}


To prove Theorem \ref{sbqua}, we will first decouple the enumerator $\langle A^{-1} X,X \rangle$ from the denominator $\sqrt{1+\V A^{-1}X \V_2^2}$ by showing that $\V A^{-1} X\V_2 \sim \V A^{-1} \V_{\HS}$ with high probability. Then we adapt argument from \cite{Invofsym} to decouple $\langle A^{-1} X,X \rangle$. Finally, by condition on $X$ we obtain a linear form, and we can estimate its small ball probabilities using the Littlewood-Offord theory.

The following result is an analog of Proposition 8.2 in \cite{Invofsym}, it compares the size of the denominator $\sqrt{1+\V A^{-1}X \V_2^2 }$ to $\V A^{-1} \V_{\HS}$.

\begin{proposition}\label{siaix}(Size of $A^{-1} X$) Let $A$ be an $n\times n$ random matrix satisfies Assumption \ref{assump}. Let $X$ be a random vector in $\R^n$ whose entries are 
identically distributed, and satisfy the same assumption as those of $A$. 
There exist constants $c_{\ref{siaix}},C_{\ref{siaix}},c_{\ref{siaix}}'>0$ that depend only on the parameter $C_{op}$ and $M_4$ from the assumption, and such that the following holds. Let $n^{-c_{\ref{siaix}}} \le \lambda \le c_{\ref{siaix}}$. The random matrix $A$ has the following property with probability at least $1-\exp(-c_{\ref{siaix}}np)$. If $C_{op}$ holds, then for every $\vv >0$, one has:

\begin{flushleft}
(i) with probability of at least 
$1-\exp(-c_{\ref{siaix}}' pn)$ in $X$, we have 
\end{flushleft}
$$
\V A^{-1} X \V_2 \ge C_{\ref{siaix}}^{-1};
$$
(ii) with probability at least $1-\vv$ in $X$, we have 
$$
\V A^{-1} X \V_2 \le \sqrt{p} \vv^{-1/2} \V A^{-1} \V_{\HS};
$$
(iii) with probability at least $1- C_{\ref{siaix}}\vv/\sqrt{\lambda} - n^{c_{\ref{siaix}}'/\lambda}$
in $X$, we have 
$$
\V A^{-1} X\V_2 \ge \sqrt{p} \vv \V A^{-1} \V_{\HS}.
$$
And the same result of (iii) still hold if we replace $X$ by an i.i.d. random vector with $\L (X_i,\vv_0) \le 1 - c_0p$. In this case $C_{\ref{siaix}}, c_{\ref{siaix}}'$ would also depend on $p_0,\vv_0$.
\end{proposition}

The proof of this result uses the following elementary lemma, see Lemma 8.3 in \cite{Invofsym}.
\begin{lemma}\label{sumind}{\rm (Sums of dependent random variables.)}
Let $Z_1,\cdots, Z_n$ be arbitrary non-negative random variables (not necessarily independent), and $p_1,\cdots,p_n$ be non-negative numbers such that 
$$
\sum_{k=1}^n p_k =1. 
$$
Then for every $\vv \in \R$ one has
$$
\P \left\{ \sum_{k=1}^n p_k Z_k \le \vv \right\}
<2 \sum_{k=1}^n p_k \P\left\{ Z_k \le 2\vv  \right\}.
$$

\end{lemma}

\begin{proof}[Proof of Proposition \ref{siaix}]
Denote $e_1,\cdots,e_n$  the canonical basis of $\R^n$, and
$$
x_k := \frac{A^{-1}e_k}{\V A^{-1}e_k \V_2},k=1,\cdots,n.
$$

Now, apply Structure Theorem \ref{stsp} together with a union bound over $k=1,\cdots,n$. More specifically, choose $L=L_0=(\delta_0 p)^{
-1/2}$ (the choice of $\delta_0$ see remark \ref{delta0}). The random matrix with probability at least $1-n2e^{-c_{\ref{stsp}}'pn}\ge 1-2 e^{-c_{\ref{stsp}}'pn/2}$ has following property: if $\e_{op}$ holds then 
$$
x_k \in \incomp(c_sn,c_d), \hat{D}_L (x_k,\lambda )\ge L^{-2} n^{c_{\ref{stsp}}/\lambda}, k=1,2,\cdots,n.
$$
From now on, let us fix a realization of $A$ satisfies above property. Without loss of generality, we may also assume that $\e_{op}$ holds.

(i) First, we have
$$
\V X \V_2 \le \V A \V \V A^{-1} X\V_2.
$$
By the definition of event $\e_{op}$, we have $\V A \V \le C_{op} \sqrt{pn}$. Moreover, Chernoff's inquality together with the Tensorization Lemma \ref{tens} implies that the random vector $X$ satisfies $\V X \V_2 \ge c\sqrt{pn}$ with probability at least $1-\exp(-cpn)$. Here $c$ is a constant only depending on $M_4$. Then we have $\V A^{-1} X\V_2 \ge \frac{c}{C_{op}}$ with the same probability. So we proved (i).

(ii) Using the fact that $A$ is symmetric, we have
$$
\V A^{-1} X\V_2^2 =\sum_{k=1}^n \langle A^{-1} X,e_k \rangle^2 
=\sum_{k=1}^n \langle A^{-1} e_k,X \rangle^2 
=\sum_{k=1}^n \V A^{-1} e_k\V_2^2 \langle x_k,X \rangle^2. 
$$
Recall that we also have $X_i=\delta_i \xi_i$, where $\delta_i$s are Bernoulli with parameter $p$ and $\xi_i$s are random variables with mean 0 variance 1. Therefore, 
$$
\E_X \langle x_k, X \rangle^2 
= \E_X \sum_{i=1}^n x_{k,i}^2 X_i^2= p
$$
So,
$$
\E_X \V A^{-1}X \V_2^2 
= \sum_{k=1}^n p\V A^{-1} e_k\V_2^2 =p\V A^{-1} \V_{\HS}^2.
$$
Part (ii) follows directly from an application of Markov's inequality.

(iii) Now, we fix $k\le n$. Then $\langle x_k, X\rangle $ is a sum of independent random variables: $\sum_{i=1}^n x_{k,i}X_i$. We can estimate this sum using Proposition \ref{sbrlcd2} combined with the estimated on the regularized LCD of $x_k$. Therefore
\begin{equation}\label{sbinner}
\L\left(
\langle x_k,X \rangle  , \sqrt{2p} \vv
\right)
\le C_{\ref{sblcd2}}\left(
\frac{\vv}{\sqrt{\lambda}} +
p^{-1/2}L^2 n^{-c_{\ref{stsp}}/\lambda}
\right).
\end{equation}
Now, together with estimates for all $k$ using (\ref{sbinner}), Lemma \ref{sumind}  with $p_k = \V A^{-1} e_k\V_2^2 / \V a^{-1} \V_{\HS}^2$ and that $\sum p_k =1$. We have
\begin{equation}
\begin{array}{rl}
\P_X \left\{
\V A^{-1} X \V_2 \le \vv \sqrt{p} \V A^{-1} \V_{\HS} 
\right\}
& =\displaystyle\P\left\{
\sum_{k=1}^n p_k \langle x_k,X \rangle^2 \le p\vv^2 
\right\}\\
& \le 2 \displaystyle\sum_{k=1}^n p_k 
\P \left\{ \langle x_k,X \rangle^2 \le 2p\vv^2 \right\}\\
& \le 2 C \displaystyle\left(
\frac{\vv}{\sqrt{\lambda}} +
p^{-3/2}n^{-c_{\ref{stsp}}/\lambda}
\right)
\end{array}
\end{equation}
We complete the proof using the range of $\lambda$ and $p$. To prove the same result hold for $X$ replaced by an i.i.d. random vector with $\L (X_i,\vv_0) \le 1 - c_0p$. We only need to notice that to derive (\ref{sbinner}) from Theorem \ref{sbrlcd}, above condition is sufficient.
\end{proof}

Decoupling the quadratic form is based on the following Lemma, see Lemma 8.4 in \cite{Invofsym}.

\begin{lemma}\label{declemma}{\rm (Decoupling quadratic forms)}. Let $G$ be an arbitrary symmetric $n\times n$ matrix, and let $X$ be a random vector in $\R^n$ with independent coordinates. Let $X'$ denote an independent copy of $X$. Consider a subset $J\subset [n]$. Then for every $\vv \ge 0$, one has
\begin{equation}
\begin{array}{rl}
\L(\langle GX,X \rangle,\vv )^2 & =
\displaystyle\sup_{u\in \R} \P \left\{
|\langle GX,X \rangle -u | \le \vv 
\right\}^2\\
& \le \displaystyle\P_{X,X'} \left\{
| \langle G (P_{J^c}(X-X')),P_J X  \rangle - v|\le \vv
\right\}
\end{array}
\end{equation}
where $v$ is some random variable whose value is determined by the $J^c \times J^c$ minor of $G$ and the random vectors $P_{J^c}X,P_{J^c} X'$.

\end{lemma}

Now, we are ready to prove Theorem \ref{sbqua}. The argument is based on the decoupling lemma and Littlewood-Offord theory which stated earlier. The proof is a modification of Section 8.3 in \cite{Invofsym}. Although the proof structure is the same as in \cite{Invofsym}, we still need to go into details to catch the effect of sparsity. 

\vspace{0.5cm}
{\bf Step 1: Constructing a random subset $J$ and assignment spread$(x)$.}
We start by decomposing $[n]$ into two random sets $J$ and $J^c$. To this end, we consider independent ${0,1}-$-valued random variables $\gamma_1,\cdots,\gamma_n$. with $\E \gamma_i =c_{oo}/2$. We also define
$$
J:=\{ i\in [n]: \gamma_i =0 \}
$$
Then $\E |J^c| =c_{oo}n/2$. By a in large deviation inequality (\cite{probmeth} Theorem A.1.4), the inequality
\begin{equation}\label{cond1}
|J^c| \le c_{oo}n
\end{equation}
holds with high probability:
$$
\P \left\{ (\ref{cond1}) {\rm \ holds} \right\} \ge 1- 2\exp(-c_{oo}' pn)
$$
where $c_{oo}'= c_{oo}^2/2$.

Fix a realization of $J$ that satisfies (\ref{cond1}). By Lemma \ref{incompsp}, at least $2c_{oo}n$ coordinates of a vector $x\in \incomp(c_sn,c_d)$ satisfy the regularity condition. So for each vector $x\in \incomp(c_sn,c_d)$ we can assign a subset
$$
{\rm spread}(x) \subseteq J,\ |{\rm spread}(x)|=\left \lceil{c_{oo}n}\right \rceil 
$$
so that the regularity condition holds for all $k\in $spread$(x)$. If there is more than one way to assign spread$(x)$ to $x$, we only need to  choose one fixed way. This results in an assignment that depends only on the choice of the random set $J$. We will use this specific assignment $J$ in applications of Definition \ref{rLCDdef1} for regularized LCD.

\vspace{0.5cm}
{\bf Step 2. Estimating the denominator $\sqrt{1+\V A^{-1}\V_2^2}$ and LCD of the inverse.}
By Lemma \ref{siaix}, we may replace the denominator $\sqrt{1+\V A^{-1} X \V_2^2}$ by $\V A^{-1} \V_{\HS}$ in (\ref{lvqua}). Let $\vv_0 \in (0,1)$ and let $X'$ denote an independent copy of the random vector $X$. Then we consider following event which is determined by the random matrix $A$, random vectors $X,X'$ and the random set $J$:
\begin{equation}\label{cond2}
\sqrt{\vv_0 p^{-1}}\sqrt{1+\V A^{-1} X\V_2^2}
\le \V A^{-1} \V_{\HS} \le \frac{1}{\vv_0 \sqrt{p}} 
\V A^{-1} (P_{J^c} (X-X'))\V_2
\end{equation}
Denote $Y:=P_{J^c}(X-X')$, then we have $Y_i$s are i.i.d. random variables and 
$$
\L (Y_i,c_0) \le \L (P_{J^c}X,c_0) \le 1 - c_1p.
$$
where $c_0,c_1$ depends only on $M_4,C_{op}$. Here we simply used the fact that $P_{J^c}X$ is a sparse random variable with sparsity level $c_{oo}p/2$ and Remark \ref{delta0}. So we can apply Proposition \ref{siaix} with $A^{-1}X$ and $A^{-1}Y$. We have 
$$
\P_{A,X,X',J} \{
(\ref{cond2}){\rm \ and \  holds \ } \wedge \e_{op}^c 
\}
\ge 1 - \frac{C_{\ref{siaix}} \vv_0}{\sqrt{\lambda}} - n^{-c_{\ref{siaix}}'/\lambda } - 2e^{-c_{\ref{siaix}}'pn}.
$$
where $c_{\ref{siaix}}',C_{\ref{siaix}}$  depend only on $C_{op}$ and $M_4$.

Denote the random vector 
$$
x_0:= \frac{A^{-1} (P_{J^c} (X-X'))}{\V A^{-1} (P_{J_c} (X-X')) \V_2}
$$
and condition on an arbitrary realization of random vectors $X,X'$ and on realization of $J$ which satisfies (\ref{cond1}). Fix a value of  parameter $\lambda$ that satisfying $n^{-c_{\ref{stsp}}} \le \lambda \le \frac{c_{\ref{stsp}}}{4}$ as needed in Theorem \ref{stsp}. Then consider the event
\begin{equation}\label{cond3}
x_0 \in \incomp(c_sn,c_d) \ {\rm and \ } 
\hat{D}_{L_0} (x_0,\lambda ) \ge \delta_0 p n^{c_{\ref{stsp}}/\lambda}
\end{equation}
which depends on the random matrix $A$. By Theorem \ref{stsp}, we have
$$
\P_A \left\{
(\ref{cond3}) {\rm \ holds \ } \vee \e_{op}^c |
X, X', J {\rm satisfies\ } (\ref{cond1})
\right\}\ge 1 -2e^{-c_{\ref{stsp}}'pn}.
$$
Therefore
\begin{equation}
\begin{array}{rl}
& \P_{A,X,X',J} \left\{ (\ref{cond1},\ref{cond2},\ref{cond3}) {\rm \ hold\ } \vee \e_{op}^c \right\} \\
\ge & \displaystyle 1-2e^{-c_{oo}'n}  - \frac{C_{\ref{siaix}} \vv_0}{\sqrt{\lambda}} - n^{-c_{\ref{siaix}}'/\lambda } - 2e^{-c_{\ref{siaix}}'pn} -2 e^{-c_{\ref{stsp}}'pn}\\
=: & 1-p_0
\end{array}
\end{equation}
Thus there exists a realization of $J$ that satisfies (\ref{cond1}) and 
$$
\P_{A,X,X'} \{ (\ref{cond2},\ref{cond3}) {\rm \ hold} \vee \e_{op}^c  \} \ge 1-p_0.
$$

Now, fix such a realization of $J$ in the rest of the proof. Applying Fubini's theorem and we have $A$ has the following property with probability at least $1-\sqrt{p_0}$:
$$
\P_{X,X'} \{ (\ref{cond2},\ref{cond3}) {\rm \ hold \ } \vee \e_{op}^c | A \}
\ge 1-\sqrt{p_0}
$$
Since event $\e_{op}^c$ depends on $A$ only, the random matrix $A$ has the following property with probability at least $1-\sqrt{p_0}$. Either $\e_{op}^c$ holds, or:
\begin{equation}\label{cond4}
\e_{op} {\rm \ holds \ and \ } 
\P_{X,X'} \{ (\ref{cond2}), (\ref{cond3}) {\rm \ hold} |A \} \ge 1-\sqrt{p_0}
\end{equation}

\vspace{0.5cm}
{\bf Step 3: Decoupling. }
Recall the event we want to estimate probability is
$$
\e := \left\{ 
\frac{| \langle A^{-1} X,X \rangle -u | }{\sqrt{1+\V A^{-1} X \V_2^2}} \le \vv\sqrt{p}
\right\}.
$$
So we only need to estimate
$$
\P_{A,X} (\e \cap \e_{op})
\le \P_{A,X} \{ \e \wedge {\rm \ (\ref{cond4})\ holds} \}
+ \P_{A,X} \{ \e_{op} \wedge  {\rm (\ref{cond4})\ fails}\}
$$
The second term is bounded by $\sqrt{p_0}$. Therefore,
$$
\P_{A,X} (\e \cap \e_{op})
\le \sup_{A {\rm \ satisfies \  (\ref{cond4})}} 
\P_X(\e |A) +\sqrt{p_0}
$$
Moreover,  using property (\ref{cond4}) in a  larger probability space, we have
$$
\P_{A,X} (\e \cap \e_{op})
\le \sup_{A {\rm \ satisfies \  (\ref{cond4})}}
\P_{X,X'} \{ \e \wedge {\rm \ (\ref{cond4})\ holds} |A \}
+2\sqrt{p_0}
$$
Now, we fix a realization of a random matrix $A$ satisfying (\ref{cond4}). We only need to bound the probability
$$
p_1 :=\P_{X,X'} \{ \e \wedge {\rm (\ref{cond2}) \ holds} \}
$$
By definition of $\e$ and property (\ref{cond2}), 
$$
p_1 \le 
\P_{X,X'} \left\{
|\langle A^{-1}X,X \rangle -u| \le \frac{p\vv}{\sqrt{\vv_0}} \V A^{-1} \V_{\HS}
\right\}
$$
Now we may apply decoupling Lemma \ref{declemma}, and therefore
$$
p_1^2 \le \P_{X,X'} \{\e_0\}
$$
where 
$$
\e_0=\left\{
|\langle A^{-1}(P_{J^c}(X-X')), P_J X \rangle -v |
\le \frac{\vv p }{\sqrt{\vv_0}} \V A^{-1} \V_{\HS}
\right\}.
$$
Here $v$ is a number that depends on $A^{-1}$, $P_{J^c} X, P_{J^c}X'$ only. Use property (\ref{cond4}) and we have
$$
p_1^2 \le \P_{X,X'} \{\e_0\}
\le \P_{X,X'} \left\{ 
\e_0 \wedge {\rm (\ref{cond2},\ref{cond3}) \ hold}
\right\}+\sqrt{p_0}
$$
Now, we may divide both sides in the inequality defining the event $\e_0$ by $\V A^{-1}(P_{J^c}(X-X')) \V_2$. By definition of $x_0$ and (\ref{cond2}) and we get
\begin{equation}\label{probp1}
p_1^2 \le \P_{x,X'}
\left\{
|\langle x_0,P_JX \rangle -w |\le \sqrt{p}\vv_0^{-3/2} \vv \wedge {\rm (\ref{cond3}) \ holds}
\right\}+\sqrt{p_0}
\end{equation}
where $w=w(A^{-1},P_{J^c}X,P_{J^c}X')$ is a  number.

\vspace{0.5cm}
{\bf Step 4: The small ball probabilities of a linear form.}
Finally, the random vector $x_0$ depends only on $P_{J^c}(X-X')$, which is independent of the random vector $P_J X$. So we may fix an arbitrary realization of the random vectors $P_{J^c}X$ and $P_{J_c}X'$, this will fix vector $x_0$ and number $w$ in (\ref{probp1}). By (\ref{cond3}) we have
$$
p_1^2\le \sup_{x_0{\rm  \ satisfies \ (\ref{cond3}), \ w\in\R}}
\P_{P_J X} \left\{
|\langle x_0,P_J X\rangle -w| \le \sqrt{p}\vv_0^{-3/2}\vv
\right\}+\sqrt{p_0}
$$
So from now on, let us fix a vector $x_0\in S^{n-1}$ such that (\ref{cond3}) holds and a number $w\in \R$. This reduce the problem to estimating the small ball probability for the weighted sum of independent random variables
$$
\langle x_0,P_J X \rangle =\sum_{k\in J} x_{0,k} X_k.
$$
We now apply Proposition \ref{sbrlcd2}, noticing that we have $J\supseteq {\rm spread}(x_0)\supseteq I(x)$ as needed in the theorem. Therefore
$$
\P_{P_J X}\left\{
|\langle x_0,P_J X\rangle -w| \le \sqrt{p}\vv_0^{-3/2}\vv
\right\} 
\le \frac{C_{\ref{sbrlcd2}}\vv_0^{-3/2}\vv}{\sqrt{\lambda}} + \frac{C_{\ref{sbrlcd2}}}{\sqrt{p}\hat{D}_{L_0}(x_0,\lambda)}.
$$
Using property (\ref{cond3}) to bound the regularized LCD, we have 
$$
p_1^2 \le \frac{C_{\ref{sbrlcd2}}\vv_0^{-3/2}\vv}{\sqrt{\lambda}} + C_{\ref{sbrlcd2}}\delta_0^{-1}p^{-3/2}n^{-c_{\ref{stsp}}/\lambda}+\sqrt{p_0}.
$$
Now we set $\vv_0 = \vv^{1/2}/\lambda^{1/8}$ and estimate $\P_{A,X}(\e \cap \e_{op})$ as
\begin{equation}
\begin{array}{rl}
\P_{A,X}(\e \cap \e_{op})     \le & p_1 + 2\sqrt{p_0}  \\
     \le &  \displaystyle \left(\frac{C_{\ref{sbrlcd2}}\vv_0^{-3/2}\vv}{\sqrt{\lambda}}\right)^{1/2} + \left(C_{\ref{sbrlcd2}}\delta_0^{-1}p^{-3/2}n^{-c_{\ref{stsp}}/\lambda}\right)^{1/2} \\
     & \displaystyle + 3\left(2e^{-c_{oo}'n} + \frac{C_{\ref{siaix}} \vv_0}{\sqrt{\lambda}} + n^{-c_{\ref{siaix}}'/\lambda } + 2e^{-c_{\ref{siaix}}'pn} +2 e^{-c_{\ref{stsp}}'pn} \right)^{1/4}\\
    \le & C\displaystyle \left( e^{-cpn} + n^{-c/\lambda} + \frac{ \vv_0^{1/4}}{\lambda^{1/8}}
    +\frac{\vv_0^{-3/4}\vv^{1/2}}{\lambda^{1/4}}\right)^{1/2} \\
    \le & \displaystyle n^{-c'/\lambda} + C'\frac{\vv^{1/8}}{\lambda^{5/32}} + e^{-c'pn}
\end{array}
\end{equation}
Optimizing above probability using $n^{-c_{\ref{stsp}}} \le \lambda \le \frac{c_{\ref{stsp}}}{4}$ (see page 49 and Fact 8.6 in \cite{Invofsym}), we have  
$$
\P_{A,X}(\e \cap \e_{op})\le C'' \vv^{1/9}+ \exp(-n^{c''}) + \exp(-c'np)
$$
where $c',c'', C''$ depend only on $M_4,C_{op}$.

\section{Proof of Theorem \ref{mth}}\label{SecMproof}

Now we can combine the incompressible and compressible part to prove Theorem \ref{mth}.
\begin{proof}[Proof of Theorem \ref{mth}.]
We consider 
\begin{equation}
\begin{array}{rl}
& \displaystyle\P \left\{
\min_{x\in S^{n-1}} \V A x\V_2 \le \vv \sqrt{\frac{p}{n}} \wedge \e_{op}
\right\}\\
\le & \displaystyle\P \left\{
\inf_{x\in \comp(c_sn,c_d)} \V A x\V_2 \le \vv \sqrt{\frac{p}{n}} \wedge \e_{op}
\right\}\\
&+\displaystyle\P \left\{
\inf_{x\in \incomp(c_sn,c_d)} \V A x\V_2 \le \vv \sqrt{\frac{p}{n}} \wedge \e_{op}
\right\}\\
\end{array}
\end{equation}
The first term is bounded by $2\exp(-c_{\ref{sbcomp}}pn)$ as in (\ref{lbforcomp}). The probability for the incompressible vectors is estimated via distances in Lemma \ref{invviadist}. Finally, we only need to apply Theorem \ref{sbqua} and Proposition \ref{distqua}, and notice that $e^{-n^{c_{\ref{sbqua}}}}$ dominate the
term $e^{-cpn}$ for $p\ge n^{-c_p}$.
\end{proof}

\section{Estimate of the Spectral Norm}\label{SecSpectNorm}

In this section, we prove Theorem \ref{upbdforop}, that is to say when $\xi_{ij}$ is sub-gaussian, $\V A \V \le C\sqrt{np} $ w.h.p.. The proof use the same moment technique and structure as the proof of Theorem 1.7 in  \cite{InvofsparsenonHerm}.

\begin{proof}[Proof of Theorem \ref{upbdforop}]
First, let's consider $\xi_{ij}'$, $i,j\in [n]$ to be independent copies of $\xi_{ij}$, $i,j\in [n]$ and $\eta_{ij}:=\xi_{ij}-\xi_{ij}'$. Let $A_n'$ and $B_n$ be the matrices with entries $a_{ij}'=\delta_{ij}\xi_{ij}'$ and $b_{ij}=\delta_{ij}\eta_{ij}$ . Denote  $\E_{\xi}$ as the expectation with respect to $\xi$, conditioned on $\bm{\delta}:=(\delta_{ij})_{i,j\in[n]}$. Consider $q\ge 1$ to be an even integer. By Jensen's inequality, as operator norm is convex function of matrix entries, we have
$$
\E_{\xi} \V A_n \V^q = \E_{\xi} \V A_n - \E_{\xi'} A_n' \V^q \le \E_{\eta}\V B_n \V^q.
$$
Then, let $g_{ij}$, $i,j\in [n]$ be independent $N(0,1)$ random variables. Clearly, $\xi_{ij}-\xi_{ij}'$ is a sub-gaussian random variable, by moment condition of sub-gaussian random variable  there exists a constant $C_1$, depending on the sub-guassian norm of $\xi_{ij}$, such that $\E |\eta_{ij}|^q \le \E |C_1g_{ij}|^q$ for all $q\ge 1$. Let $W_n$ be the $n\times n$ random matrix with entries $w_{ij}=\delta_{ij}g_{ij}$. Since 
$$
\E_{\eta} \V B_n \V^q \le \E_{\eta} {\rm Tr} \left( (B_n B_n^*)^{q/2} \right)
$$
where right hand side is a polynomial of the even moments of $\eta_{ij}$ with non-negative coefficients, we have
$$
\E_{\eta}{\rm Tr} \left( (B_n B_n^*) ^{q/2} \right) \le 
C_1^q n \E_g \V W_n \V^q.
$$
Above inequality uses the elementary identity that $\tr \left( (W_n W_n^*)^{q/2} \right) = \sum_{j=1}^n \lambda_j^{q/2} (W_n W_n^*) $. Here eigenvalues $\lambda_j(W_n W_n^*)$ satisfy $|\lambda_j (W_n W_n^*)|\le \V W_n \V^2$ for all $j$.

Now we are ready to estimate $\E \V W\V^2$. Here we need to apply the following result due to Bandeira and van Handel \cite{opbdHandel}.
\begin{lemma}\label{opbdhandel}
Let $X$ be the $n\times n$ symmetric matrix with $X_{ij}=g_{ij}b_{ij}$, where $\{ g_{ij}: i\ge j\}$ are i.i.d.$\sim N(0,1)$ and $\{ b_{ij} : i\ge j \}$ are given scalars. Let
$$
\sigma := \max_{i} \sqrt{\sum_j b_{ij}^2},\ \ \sigma_{*}:= \max_{ij} |b_{ij}| 
$$

Then 
$$
\E \V X \V \le (1+\vv) \left\{2 \sigma +\frac{6}{\sqrt{\log (1+\vv)}}\sigma_{*} \sqrt{\log n} \right\}
$$
for any $\vv \in (0,1/2)$.
\end{lemma}

Let $\Omega$ be the event for all $i\in [n]$, $\sum_{j=1}^n \delta_{ij} \le \bar{C}pn$, for some $\bar{C}\ge 2$. Since $p\ge C_0 \frac{\log n}{n}$, applying Chernoff's inequality and union bound argument,  we can choose the $C_0$ large enough, such that $\P (\Omega^c) \le e^{-cpn}$ for some $c>0$. And $c$ depends only on $C_0$. Now, we can use the above Lemma \ref{opbdhandel} and assume that $\bm{\delta}\in \Omega$. Conditionally on $\bm{\delta}$, we have
$$
\E \left( \V W_n \V | \bm{\delta} \right)
\le \sqrt{\bar{C} pn} +C^* \sqrt{\log n} \le \sqrt{C' pn}.
$$
Here $C^*$ is some absolute constnat,  and $C'=2(C^*)^2 \bar{C}$. Conditioning on $\bm{\delta}$, $\V W_n \V$ can be viewed as a $\sqrt{2}$-Lipschitz function on $\R^{n(n+1)/2}$  with the standard Gaussian measure. Applying standard Gaussian concentration inequality \cite{concofmea}, we have
$$
\P \left( \V W_n \V \ge \E [\V W_n \V |\bm{\delta} ] + t \right) \le \tilde{C}\exp(-c't^2)
$$
for some absolute constants $\tilde{C},c'>0$, and any $t>0$. Therefore,
\begin{equation}
\begin{array}{rl}
 \E_g  \V W_n \V^q  & \le   (C'pn)^{q/2} + \displaystyle\int_{\sqrt{C'pn}}^{\infty} q s^{q-1}  
 \P \left[ \V W_n \V \ge s |\bm{\delta} ] \right] ds \\
     &  \le (C'pn)^{q/2} + (C''q)^{q/2},
\end{array}
\end{equation}
for some absolute constant $C''$. Now choose $q=pn$. This inequality in combination with previous inequalities yields
$$
\E_{\xi} \V A_n \V^{pn} \le n (C_2pn)^{pn/2}\le (C_2^2 pn)^{pn/2}.
$$
where $C_2$ is a positive constant depending on $C_0$ and the sub-gaussian norm of ${\xi_{ij}}$. Here we used the condition $p\ge C_0 \frac{\log n}{n}$ to absorb the factor $n$.  Finally, choosing $C_{op} >C_2^2$, we have  for any $\bm{\delta}\in \Omega$, there exists a small positive constant $c_{op}$, depending on $C_{op}$, such that 
$$
\P \left( \V A_n \V \ge C_{op} \sqrt{pn} |\bm{\delta}  \right)
\le \exp(-c_{op}pn)
$$
by applying Markov inequality. Now picking  $c_{op}$ small enough, we have 
$$
\P \left( \V A_n \V \ge C_{op} \sqrt{pn} \right)
\le \max_{\bm{\delta}\in \Omega} 
\P \left[ \V A_n \V \ge C_{op} \sqrt{pn} |\bm{\delta}  \right]
+\P (\Omega^c) \le \exp(-c_{op}pn).
$$

\end{proof}
\renewcommand\refname{Reference}
\bibliographystyle{plain}
\bibliography{feng}

\begin{thebibliography}{10}

\bibitem{opbdHandel}
R.~Handel A.~Bandeira.
\newblock Sharp nonasymptotic bounds on the norm of random matrices with
  independent entries.
\newblock {\em Annals of Probability}, 44(2479-2506), 2016.

\bibitem{probmeth}
N.~Alon and J.~Spencer.
\newblock {\em The probability method}.
\newblock Wiley-Interscience, New York, 2000.

\bibitem{InvofsparsenonHerm}
A.~Basak and M.~Rudelson.
\newblock Invertibility of sparse non-hermitian matrices.
\newblock {\em arXiv:1507.03525}.

\bibitem{AEecnrm}
A.~Edelman.
\newblock Eigenvalues and condition numbers of random matrices.
\newblock {\em SIAM J. Matrix Anal. Appl.}, 9(543-560), 1988.

\bibitem{concofmea}
M.~Ledoux.
\newblock {\em The concentration of measure phenomenon}.
\newblock American Mathematical Soc., 2005.

\bibitem{InvofRM}
M.Rudelson.
\newblock Invertibility of random matrices: norm of the inverse.
\newblock {\em Annals of Mathematics}, 168(575-600), 2008.

\bibitem{Invofheavytail}
E.~Rebrova and K.~Tikhomirov.
\newblock Coverings of random ellipsoids, and invertibility of matrices with
  i.i.d. heavy-tailed entries.
\newblock {\em arXiv:1508.06690}.

\bibitem{LecnotesRMMR}
M.~Rudelson.
\newblock Lecture notes on non-asymptotic random matrix theory.
\newblock {\em notes from the AMS Short Course on Random Matrices}, 2013.

\bibitem{LOandInvofRM}
M.~Rudelson and R.~Vershynin.
\newblock The {L}ittlewood-{O}fford problem and invertibility of random
  matrices.
\newblock {\em Advances in Mathematics}, 218(600-633), 2008.

\bibitem{SMstSvRect}
M.~Rudelson and R.~Vershynin.
\newblock The smallest singular value of a rectangular random matrix.
\newblock {\em Communications on Pure and Applied Mathematics}, 62(1707-1739),
  2009.

\bibitem{NAtheoryofRM}
M.~Rudelson and R.~Vershynin.
\newblock Non-asymptotic theory of random matrices: extreme singular values.
\newblock {\em Proceedings of the International Congress of Mathematicians},
  Volume III(1576-1602), 2010.

\bibitem{AGA}
A.~A.~Giannopoulos S.~Artstein-Avidan and V.~D. Milman.
\newblock {\em Asymptotic geometric analysis. Part I}.
\newblock 2015.

\bibitem{SSeaa}
S.~Smale.
\newblock On the efficiency of algorithms of analysis.
\newblock {\em Bull. Amer. Math. Soc.}, 13(87-121), 1985.

\bibitem{SScnrm}
S.~Szarek.
\newblock Condition numbers of random matrices.
\newblock {\em J. Complexity}, 7(131-149), 1991.

\bibitem{TaoVuAnnals}
T.~Tao and V.~Vu.
\newblock Inverse {L}ittlewood-{O}fford theorems and the condition number of
  random discrete matrices.
\newblock {\em Annals of Mathematics}, 169(595-632), 2009.

\bibitem{TVsmallestSVdist}
T.~Tao and V.~Vu.
\newblock Random matrices: The distribution of the smallest singular values.
\newblock {\em Geometric and Functional Analysis}, 20(260-297), 2010.

\bibitem{RVbookintrotorm}
R.~Vershynin.
\newblock {\em Introduction to the non-asymptotic analysis of random matrices,
  Compressed sensing, 210-268}.
\newblock Cambridge University Press, 2012.

\bibitem{Invofsym}
R.~Vershynin.
\newblock Invertibility of symmetric random matrices.
\newblock {\em Random Structures and Algorithms}, 44, 2014.

\end{thebibliography}

\end{document}